\documentclass[final,hidelinks,onefignum,onetabnum]{siamart251216}
\usepackage{amsfonts,amssymb,mathtools,graphicx,booktabs,needspace,placeins}
\newsiamremark{assumption}{Assumption}
\newsiamremark{remark}{Remark}
\newcommand{\R}{\mathbb R}
\newcommand{\dd}{\,\mathrm d}

\allowdisplaybreaks
\headers{Nonlinear Stability with General Motility}{X. Cheng, X. Lai, and X. Song}
\title{Nonlinear Stability of Nonconstant Steady States in
Chemotaxis--Consumption Systems with General Motility}
\author{Xinyu Cheng\thanks{Research Institute of Intelligent Complex Systems,
Fudan University, Shanghai, P.R. China (\texttt{xycheng@fudan.edu.cn}).}
\and Xuewen Lai\thanks{Research Institute of Intelligent Complex Systems,
Fudan University, Shanghai, P.R. China (\texttt{xwlai26@m.fudan.edu.cn}).}
\and Xu Song\thanks{School of Mathematics, South China Normal University,
Guangzhou, Guangdong, P.R. China(\texttt{songxu@m.scnu.edu.cn}).}}
\begin{document}
\maketitle
\begin{abstract}
We prove nonlinear stability of positive nonconstant steady states for a
chemotaxis--consumption system with nutrient-dependent motility and a fixed
boundary nutrient concentration $b$. The motility satisfies two monotonicity
conditions met by power laws and saturating functions. For each
prescribed population mass, sufficiently large nutrient diffusivity $D$ yields a
unique positive steady state. Small compatible perturbations generate unique
global solutions converging exponentially to the steady state with their
conserved mass. The result holds on general bounded domains in two and three
dimensions, with a lower-regularity version on an interval. At order
$D^{-1}$, the nutrient deficit is the torsion function multiplied by the
mean population density and boundary nutrient value; the population contrast
also depends on $\kappa_b=b\phi'(b)/\phi(b)$. Computations illustrate
relaxation from compatible data and convergence to different steady states
when the initial mass changes.
\end{abstract}

\section{Introduction}\label{sec:intro}
Consider the chemotaxis--consumption system
\begin{equation}\label{eq:model}
 u_t=\Delta(u\phi(w)),\qquad w_t=D\Delta w-uw,
 \qquad \partial_\nu(u\phi(w))=0,\quad w|_{\partial\mathcal O}=b>0.
\end{equation}
Here $u$ is the population density, $w$ the nutrient concentration, and
$D>0$ the nutrient diffusivity. The domain $\mathcal O$ is an interval or a
bounded domain in two or three dimensions. The boundary conditions conserve
population mass and prescribe a nutrient reservoir at the boundary.
Integrating the population equation gives
\[
 \int_{\mathcal O}u(x,t)\dd x=\int_{\mathcal O}u(x,0)\dd x.
\]
At a positive equilibrium, $U\phi(W)=Z$ is constant and
$D\Delta W=UW>0$. Thus $W$ lies below its boundary value in the
interior, while $U=Z/\phi(W)$ is nonincreasing as a function of $W$.
For strictly increasing motility, the population is denser where
nutrient is lower and dispersal is slower.

Signal-dependent dispersal extends the Keller--Segel framework
\cite{KellerSegel1971}. Starvation-driven diffusion \cite{ChoKim2013} and
the metric-of-food model \cite{ChoiKim2015} relate movement to available
resources. Density-suppressed motility provides another feedback mechanism
\cite{LiuScience2011,FuTang2012}; analytical results for signal-dependent
motilities include \cite{JiangLaurencot2021,JinWang2020}.
In \eqref{eq:model},
the population flux is
\[
 -\nabla(u\phi(w))=-\phi(w)\nabla u-u\phi'(w)\nabla w.
\]
We consider motility increasing with nutrient on $(0,b]$, so the second
term points towards lower nutrient concentrations. Our assumptions also
require $s/\phi(s)$ to be nondecreasing there. Equivalently,
$0\le s\phi'(s)\le\phi(s)$: the relative increase in motility is no
larger than the relative increase in nutrient. Power laws $s^\alpha$,
$0<\alpha\le1$, describe proportional or sublinear responses, while
$1-e^{-s}$ and $s/(1+s)$ describe saturation at high concentrations.
The estimates below depend on these monotonicity assumptions and on
local derivative bounds for $\phi$.

Monotonicity of
$s/\phi(s)$ permits comparison for the stationary nutrient equation,
whereas monotonicity of $\phi$ makes the population mass increase along
the resulting family of profiles. During the evolution, the estimates
use the values and derivatives of $\phi$ on a fixed positive nutrient
interval maintained by the boundary supply.

The consumption model and its variants are surveyed in
\cite{LankeitWinklerSurvey}. Under homogeneous no-flux conditions, global
existence, boundedness, and relaxation have been studied in
\cite{LiZhao2021,LiWangPan2021,LiWinkler2023,LaurencotHomogeneity}.
Tao and Winkler \cite{TaoWinkler2023} constructed global very weak--strong
solutions for inverse-power-type motilities in arbitrary dimensions and
obtained stronger solutions under additional restrictions. Their interval
result allows general positive $C^3$ motilities. For degenerate local sensing,
nutrient depletion can lead to loss of population diffusivity
\cite{WinklerDegenerate2024,WinklerMaximum2024,Laurencot2024}.
When $D$ exceeds the threshold in Section~\ref{sec:main}, the steady
nutrient profile has a positive lower bound, and the nutrient concentration
remains in a common positive interval for small compatible perturbations. Small-data relaxation and one-dimensional large-data decay
also arise in related hyperbolic--parabolic systems
\cite{LiLiZhao2011,LiPanZhao2015}.

For prescribed nutrient boundary values, stationary states and boundary-layer
profiles have been investigated in
\cite{BraukhoffLankeit2019,LeeWangYang2020,LankeitWinkler2022,HuangXueZhaoZheng2026}.
Hong and Wang \cite{HongWang2021} used an antiderivative formulation to prove
one-dimensional nonlinear stability under mixed boundary conditions.
Song, Li, and Zhang \cite{SongLiZhang2024} treated spiky steady states for
power-law motilities on the half-line by weighted energy estimates.
Li and Li \cite{LiLi2023} proved nonradial stability in general smooth domains
for the population flux $\nabla u-u\nabla v$, coupled to nutrient diffusion
and consumption. Their fixed-mass results hold in $W^{1,p}$, $p>n$, in arbitrary
dimensions, cover both sufficiently large and sufficiently small nutrient
diffusion, and allow every positive diffusivity when the boundary nutrient
concentrations are at most one. Their argument uses spectral estimates and
linearized stability. In the product-flux model \eqref{eq:model}, population
diffusivity varies with nutrient concentration and its derivatives are tied
to the drift; these coefficients must be controlled in the nonlinear energy.
For logarithmic sensing, Carrillo, Li, Wang, and Yang
\cite{CarrilloLiWangYang2026} constructed small-diffusion boundary spike layers
in general domains and obtained more detailed profile expansions and nonlinear
stability in the radial setting. Our result concerns large nutrient diffusion
and permits nonradial perturbations.

\begin{theorem}[Informal main theorem]\label{thm:simple}
For the structural class of motilities specified in Section~\ref{sec:main},
and for sufficiently large nutrient diffusion, each mass near a prescribed
positive mass admits a unique positive steady state. Small compatible
perturbations produce unique global solutions converging exponentially to
the steady state with their actual conserved mass. This holds in
$H^1\times H^3$ on general bounded domains in dimensions two and three,
and in $L^2\times H^1$ on an interval. The steady profiles and the
limiting equilibrium depend locally Lipschitz-continuously on mass and
initial data, respectively; the decay constants are uniform for nearby masses.
\end{theorem}

The precise assumptions and solution classes are given in
Theorems~\ref{thm:main} and \ref{thm:interval}. The sufficient diffusion
bounds depend explicitly on the domain volume, torsion function, and
Dirichlet Poincar\'e constant.

The Neumann potential gains two spatial derivatives over
the population perturbation. Estimating the potential and nutrient in
the same $H^3$ space therefore gives the physical topology
$H^1\times H^3$. In one dimension an antiderivative gains one derivative
and allows the lower topology $L^2\times H^1$. Both formulations keep
mass conservation in the boundary conditions for the potential.
The additional estimates required in higher dimensions are discussed in
Remark~\ref{rem:higher-dimensions}.

The proof uses three ingredients.

\noindent\textbf{(i) Stationary comparison.}
The product $U\phi(W)$ is constant at equilibrium. Eliminating $U$ reduces
the steady problem to a scalar nutrient equation with a mass constraint.
Monotonicity gives uniqueness and orders the steady profiles by mass.
Comparison with the torsion function gives a uniform positive lower bound,
which permits passage from regularized motility to the degenerate model.

\noindent\textbf{(ii) Coupled potential energy.}
For a population perturbation of zero mean, define
$-\Delta q=u-U$, $\partial_\nu q=0$, $\int q=0$.
The Neumann-inverse duality structure used in
\cite[Lemma~3.1]{TaoWinkler2023}, centered here at a nonconstant equilibrium,
keeps $\phi(w)$ inside a positive dissipative term. Coupling this identity
to the nutrient energy gives low-order decay on general domains.

\noindent\textbf{(iii) Time derivatives and elliptic recovery.}
We estimate the state and its first time derivative using their
homogeneous boundary conditions. The population equation first gives
control of the Laplacian of the potential; the nutrient equation then
recovers the nutrient derivatives. In this order, scalar elliptic
estimates bound the $H^3$ norm by the energy and the $H^4$ norm by its
dissipation. The resulting pointwise control keeps the nutrient positive,
and Sobolev product estimates bound the nonlinear remainders by the
perturbation size times the dissipation. The energy estimate therefore
closes for small compatible data.

For data of mass $\mu_0$ near a reference state $(U_m,W_m)$, the
energy argument is applied after centering at $(U_{\mu_0},W_{\mu_0})$.
The resulting population perturbation has zero mean. Lipschitz
dependence of the profiles on mass then gives both convergence to the
selected equilibrium and the residual distance from the reference state.

\begin{remark}[Main difficulties]\label{rem:intro-difficulties}
The equilibrium is spatially varying, and its coefficients enter both the
principal part and the nonlinear coupling. Direct spatial differentiation
on a curved boundary does not preserve the Neumann condition. The low-order
energy controls the nutrient perturbation only in $L^2$, which does not
ensure that the population diffusivity stays positive. The construction
must attain the initial energy although $q_t(0)$ belongs only to $H^1$,
where a normal trace need not be defined. The coefficient bounds and energy
estimates must also be uniform over nearby masses.
\end{remark}

Section~\ref{sec:main} states the assumptions and main theorem.
Sections~\ref{sec:steady}--\ref{sec:global} give the steady estimates,
the nonlinear energy argument, and the stability proof. Section~\ref{sec:numerics}
presents the numerical observations. Appendix~\ref{app:technical} contains
the steady-profile regularity and construction details, Appendix~\ref{app:interval} gives
the interval result at lower regularity, and Appendix~\ref{app:numerics}
records the numerical method, further figures, and refinement checks.

\section{Assumptions and main result}\label{sec:main}
The domain $\Omega\subset\R^n$, $n=2,3$, is bounded and connected with
$C^4$ boundary. For the interval result set $\mathcal O=I=(0,1)$;
otherwise $\mathcal O=\Omega$. We write $\|\cdot\|_p$ for the
$L^p(\mathcal O)$ norm, omit the domain from Sobolev norms, and use product
norms for pairs; thus $\|Q\|_{H^k}^2=\|q\|_{H^k}^2+\|\psi\|_{H^k}^2$
for $Q=(q,\psi)$. Set
$\langle f\rangle=|\mathcal O|^{-1}\int_{\mathcal O}f$.
When a time interval $[0,T]$ is fixed, $C H^k$ and $L^2H^k$ abbreviate
$C([0,T];H^k)$ and $L^2(0,T;H^k)$, respectively.
Constants denoted by $C$ may change between lines.

\begin{assumption}[Motility]\label{ass:phi}
The function $\phi\in C([0,\infty))\cap C^3((0,\infty))$ satisfies
$\phi(0)=0$, $\phi(s)>0$ for $s>0$, and
\begin{equation}\label{eq:structure}
 0\le s\phi'(s)\le\phi(s),\qquad 0<s\le b.
\end{equation}
For the multidimensional result, assume in addition that
$\phi\in C^5(I_b)$, where $I_b=[b/4,5b/4]$.
\end{assumption}
Condition \eqref{eq:structure} is equivalent to monotonicity of both
$\phi$ and $g(s)=s/\phi(s)$ on $(0,b]$. Besides the examples in the
introduction, it includes $\phi(s)=s^\alpha/(1+s)^\beta$ for
$0\le\beta\le\alpha\le1$, $\alpha>0$.
The $C^4$ boundary provides the $H^4$ elliptic recovery used in the
dissipation estimate. The local $C^5$ assumption is sufficient for the
$H^3$ composition bounds, including that for $\phi''(W+\psi)$.

Let $\tau$ be the Dirichlet torsion function and $\lambda_1$ the first
Dirichlet eigenvalue of $-\Delta$ on $\Omega$:
\[
 -\Delta\tau=1\quad\hbox{in }\Omega,\qquad
 \tau=0\quad\hbox{on }\partial\Omega.
\]
We shall use the coefficient bounds
\begin{equation}\label{eq:constants}
 a_b=\min_{I_b}\phi>0,\qquad
 L_b=\max_{I_b}|\phi'|,\qquad
 M(\mu)=\frac{\mu\phi(b)}{|\Omega|\phi(b/2)}.
\end{equation}
In particular, $M(\mu)$ bounds the stationary population of mass $\mu$.
The three restrictions on nutrient diffusion enter at different stages
of the proof. Stationary comparison requires
\begin{equation}\label{eq:positivity-threshold}
 D_{\rm pos}(\mu)=\frac{2\mu}{|\Omega|}\|\tau\|_\infty.
\end{equation}
The low-order energy estimate uses
\begin{equation}\label{eq:low-threshold}
 D_{\rm en}(\mu)=\frac{(M(\mu)L_b+5b/4)^2}{2a_b\lambda_1}.
\end{equation}
For the gradient estimates, define
\begin{equation}\label{eq:coupling-bound}
 K(\mu)=\frac{2}{a_b}
       \left[\left(\frac{M(\mu)L_b}{\lambda_1}\right)^2
                    +\left(\frac{5b}{4}\right)^2\right]
              +\frac{M(\mu)}{\lambda_1}.
\end{equation}
The three coupling estimates in Section~\ref{sec:high-energy} give
$K(\mu)$ as the sum of the coefficients of $\|\Delta\psi\|_2^2$.
Set
\begin{equation}\label{eq:diffusion-threshold}
 D_3(\mu)=\max\{D_{\rm pos}(\mu),D_{\rm en}(\mu),4K(\mu)\}.
\end{equation}
By continuity of $D_3$, a fixed $D>D_3(m)$ also satisfies
$D>D_3(\mu)$ for all $\mu$ sufficiently close to $m$.

The stationary state of mass $\mu$ satisfies
\begin{equation}\label{eq:steady}
 U_\mu\phi(W_\mu)=Z_\mu,\qquad D\Delta W_\mu=U_\mu W_\mu,
 \qquad W_\mu|_{\partial\Omega}=b,\qquad \int_\Omega U_\mu\dd x=\mu.
\end{equation}
Section~\ref{sec:steady} proves existence and uniqueness when
$D\ge D_{\rm pos}(\mu)$, together with $b/2\le W_\mu\le b$.

\paragraph{Compatible data.}
A pair $(u_0,w_0)\in H^1(\Omega)\times H^3(\Omega)$ is compatible if
$w_0|_{\partial\Omega}=b$ and
\begin{equation}\label{eq:physical-compatibility}
 (D\Delta w_0-bu_0)|_{\partial\Omega}=0
 \quad\hbox{in }H^{1/2}(\partial\Omega).
\end{equation}
Differentiating the fixed nutrient boundary value in time gives
\eqref{eq:physical-compatibility}; the trace is defined because
$\Delta w_0,u_0\in H^1$.
An initial normal derivative of $u_0\phi(w_0)$ is not imposed at this
regularity; the zero-flux condition holds in the parabolic trace sense
at almost every positive time.

\begin{theorem}[Nonlinear stability and mass selection]\label{thm:main}
Suppose Assumption~\ref{ass:phi} holds. Fix $m>0$ and $D>D_3(m)$.
There exist $r_m,\delta,C,\omega>0$ such that every mass
$\mu\in I_m=[m-r_m,m+r_m]\Subset(0,\infty)$ has a unique positive
steady state $(U_\mu,W_\mu)$, and the following assertions hold.
For compatible initial data, put
\[
 d_0=\|u_0-U_m\|_{H^1}+\|w_0-W_m\|_{H^3},\qquad
 \mu_0=\int_\Omega u_0\dd x.
\]
If $d_0<\delta$, then $\mu_0\in I_m$ and \eqref{eq:model} has a unique
global strong solution with $b/4\le w\le5b/4$ and
\begin{equation}\label{eq:physical-decay}
 \|u(t)-U_{\mu_0}\|_{H^1}+\|w(t)-W_{\mu_0}\|_{H^3}
 \le Ce^{-\omega t}d_0,\qquad t\ge0.
\end{equation}
More precisely, set $\psi=w-W_{\mu_0}$ and let $q$ be the normalized
Neumann solution of $-\Delta q=u-U_{\mu_0}$. On every finite interval,
\begin{equation}\label{eq:strong-class}
\begin{aligned}
 Q=(q,\psi)&\in C([0,T];H^3(\Omega)^2)\cap L^2(0,T;H^4(\Omega)^2),\\
 Q_t&\in C([0,T];H^1(\Omega)^2)\cap L^2(0,T;H^2(\Omega)^2),\\
 Q_{tt}&\in L^2(0,T;L^2(\Omega)^2).
\end{aligned}
\end{equation}
The base conditions are $\partial_\nu q=0$, $\psi=0$, and $\int q=0$;
the first time derivatives satisfy the corresponding conditions for almost
every positive time. In addition,
\begin{align}
 |\mu_0-m|+\|U_{\mu_0}-U_m\|_{H^1}
 +\|W_{\mu_0}-W_m\|_{H^3}+\|Q(0)\|_{H^3}&\le Cd_0,
 \label{eq:recentering}\\
 \|Q(t)\|_{H^3}&\le Ce^{-\omega t}d_0,
 \label{eq:potential-decay}\\
 \|u(t)-U_m\|_{H^1}+\|w(t)-W_m\|_{H^3}
 &\le Ce^{-\omega t}d_0+C|\mu_0-m|.
 \label{eq:mass-selection-reference-offset}
\end{align}
The solution map is locally Lipschitz on compatible data on each finite
time interval, and the selected equilibrium depends locally
Lipschitz-continuously on the initial data. If $u_0\ge0$, then $u(t)\ge0$.
All smallness and decay constants may be chosen uniformly for reference
masses in a sufficiently small neighborhood of $m$.
\end{theorem}

\begin{remark}[Higher dimensions]\label{rem:higher-dimensions}
The potential identity below is independent of dimension. The restriction
to $n=2,3$ enters through the $H^3$ product estimates. A higher-dimensional
extension uses $(q,\psi)\in H^s\times H^s$, hence physical perturbations
in $H^{s-2}\times H^s$, with greater regularity and the corresponding
compatibility conditions. For example, an integer $s>n/2+2$ makes
$H^{s-2}$ an algebra and embeds $H^s$ into $W^{1,\infty}$.
Such an extension would also require higher-order time estimates and
elliptic recovery.
\end{remark}

\section{Steady states and their dependence on mass}\label{sec:steady}
The following stationary estimates apply both to $\Omega$ and to
$I$. We use $\mathcal O$ for either domain and write
$D_{\rm pos}(\mu)=2\mu\|\tau\|_\infty/|\mathcal O|$.

\begin{proposition}[Stationary comparison]\label{prop:steady}
Under the structural assumptions on $\phi$, if $D\ge D_{\rm pos}(\mu)$,
there is a unique positive steady state of mass $\mu$. It satisfies
\begin{equation}\label{eq:stationary-bounds}
 0\le b-W_\mu\le\frac{\mu b}{D|\mathcal O|}\tau,
 \qquad \frac b2\le W_\mu\le b,
 \qquad 0<U_\mu\le\frac{\mu\phi(b)}{|\mathcal O|\phi(b/2)}.
\end{equation}
Suppose these diffusion bounds hold for every mass in a compact interval
$J\Subset(0,\infty)$. Then
\begin{equation}\label{eq:mass-basic-bounds}
 |Z_\mu-Z_\nu|\le\frac{\phi(b)}{|\mathcal O|}|\mu-\nu|,
 \qquad
 \|W_\mu-W_\nu\|_\infty
 \le\frac{b\|\tau\|_\infty}{D|\mathcal O|}|\mu-\nu|.
\end{equation}
For $\mu\ge\nu$ in $J$, the profiles satisfy $Z_\mu\ge Z_\nu$,
$W_\mu\le W_\nu$, and $U_\mu\ge U_\nu$.
On $\Omega$, for every finite $p>n$, the stronger regularity assumptions give
\begin{equation}\label{eq:mass-high-bounds}
 \|W_\mu-W_\nu\|_{W^{4,p}}+
 \|U_\mu-U_\nu\|_{W^{3,\infty}}\le C_J|\mu-\nu|.
\end{equation}
The profiles have uniform bounds in these spaces. On $I$, the $C^3$
assumption alone gives uniform $W^{1,\infty}$ bounds and
\begin{equation}\label{eq:mass-interval-bounds}
 \|U_\mu-U_\nu\|_2+\|W_\mu-W_\nu\|_{H^1}
 \le C_J|\mu-\nu|.
\end{equation}
\end{proposition}

\begin{proof}
At a steady state, $U\phi(W)$ is a Neumann harmonic function, hence is
constant because the domain is connected. Set
$\phi_\varepsilon=\phi+\varepsilon$ and
$g_\varepsilon(s)=s/\phi_\varepsilon(s)$ for $0<\varepsilon\le1$.
Then $g_\varepsilon(0)=0$ and
$0\le g_\varepsilon'\le1/\varepsilon$ on $(0,b]$.
For fixed $Z\ge0$, the scalar problem
\[
 -D\Delta W+Zg_\varepsilon(W)=0,\qquad W|_{\partial\mathcal O}=b,
\]
has a solution between the subsolution $0$ and supersolution $b$.
The strong maximum principle gives $W>0$. Testing the difference of two
solutions with its positive part proves uniqueness. The same test shows
that $W_{\varepsilon,Z}$ decreases with $Z$.
If $Z_2\ge Z_1$ and $r=W_{\varepsilon,Z_1}-W_{\varepsilon,Z_2}$, then
\[
 (-D\Delta+Z_1c_\varepsilon)r
 =(Z_2-Z_1)g_\varepsilon(W_{\varepsilon,Z_2}),\qquad r|_{\partial\mathcal O}=0,
\]
where $c_\varepsilon\ge0$ is the bounded divided difference of
$g_\varepsilon$. Comparison with
$(Z_2-Z_1)g_\varepsilon(b)\tau/D$ proves continuity in $Z$.
Consequently
\[
 M_\varepsilon(Z)=Z\int_{\mathcal O}
       \phi_\varepsilon(W_{\varepsilon,Z})^{-1}\dd x
\]
is continuous and strictly increasing, with $M_\varepsilon(0)=0$ and
\[
 M_\varepsilon(Z)\ge\frac{Z|\mathcal O|}{\phi(b)+\varepsilon}.
\]
It therefore takes the value $\mu$ exactly once.

For this regularized state, the mass identity and monotonicity give
\[
 Z_\varepsilon\le\frac{\mu(\phi(b)+\varepsilon)}{|\mathcal O|},
 \qquad Z_\varepsilon g_\varepsilon(W_\varepsilon)
 \le\frac{\mu b}{|\mathcal O|}.
\]
Comparison of $b-W_\varepsilon$ with
$\mu b\tau/(D|\mathcal O|)$ gives \eqref{eq:stationary-bounds}, uniformly
in $\varepsilon$, including $W_\varepsilon\ge b/2$. Also,
\[
 \frac{\mu\phi(b/2)}{|\mathcal O|}\le Z_\varepsilon
 \le\frac{\mu(\phi(b)+\varepsilon)}{|\mathcal O|},\qquad
 U_\varepsilon\le\frac{\mu\phi(b)}{|\mathcal O|\phi(b/2)}.
\]
Uniform elliptic estimates on this positive range permit passage to
$\varepsilon\downarrow0$, strongly in $C^{2,\beta}$ for some
$\beta\in(0,1)$ after decreasing $\beta$. The limit has the prescribed mass.
Any positive unregularized state satisfies the same bounds, since
$W\le b$, $Z\le\mu\phi(b)/|\mathcal O|$, and $g$ is nondecreasing.
For two such states, $Z_1\le Z_2$ implies $W_1\ge W_2$. Their equal masses
then give
\[
 Z_2\int\frac1{\phi(W_2)}
 \ge Z_2\int\frac1{\phi(W_1)}
 \ge Z_1\int\frac1{\phi(W_1)}.
\]
Equality of the endpoints forces $Z_1=Z_2$; scalar comparison gives
$W_1=W_2$, so the full regularized family converges to this state.

For $\mu\ge\nu$ in $J$, the same ordering yields $Z_\mu\ge Z_\nu$ and
$W_\mu\le W_\nu$. Monotonicity of $\phi$ then gives
\[
 U_\mu=\frac{Z_\mu}{\phi(W_\mu)}
 \ge\frac{Z_\nu}{\phi(W_\nu)}=U_\nu.
\]
Subtracting the mass identities gives
\[
 \mu-\nu\ge (Z_\mu-Z_\nu)\int\phi(W_\mu)^{-1}
 \ge\frac{|\mathcal O|}{\phi(b)}(Z_\mu-Z_\nu).
\]
With $r=W_\nu-W_\mu\ge0$ and
$c=\int_0^1g'(W_\mu+\theta r)\dd\theta$, the difference equation is
\begin{equation}\label{eq:steady-difference}
 (-D\Delta+Z_\nu c)r=(Z_\mu-Z_\nu)g(W_\mu),
 \qquad r|_{\partial\mathcal O}=0.
\end{equation}
Comparison with $(Z_\mu-Z_\nu)g(b)\tau/D$ proves
\eqref{eq:mass-basic-bounds}. The higher-order profile bounds follow by scalar elliptic estimates
applied to this difference equation; the details are given in
Appendix~\ref{app:steady-proof}.
\end{proof}

\begin{proposition}[Profiles at large nutrient diffusion]\label{prop:large-D}
Fix a compact mass interval $J\Subset(0,\infty)$ and assume the
structural and $C^3$ conditions in Assumption~\ref{ass:phi}.
Write $(U_D,W_D)$ for the steady
state of mass $m\in J$, displaying its dependence on $D$, and put
\[
 \rho=\frac{m}{|\mathcal O|},\qquad
 \kappa_b=\frac{b\phi'(b)}{\phi(b)}\in[0,1].
\]
There are $D_*,C>0$, independent of $m\in J$ and $D\ge D_*$, such that
\begin{equation}\label{eq:large-D-profiles}
 \left\|W_D-b+\frac{\rho b}{D}\tau\right\|_\infty
 +\left\|U_D-\rho-\frac{\rho^2\kappa_b}{D}
                  (\tau-\langle\tau\rangle)\right\|_\infty
 \le\frac{C}{D^2}.
\end{equation}
Here $\mathcal O$ may be $\Omega$ or $I$, with its Dirichlet torsion
function $\tau$; $\kappa_b$ is the relative sensitivity of motility at
the boundary nutrient value.
\end{proposition}

\begin{proof}
Take $D_*\ge\max\{1,\sup_{m\in J}D_{\rm pos}(m)\}$.
Proposition~\ref{prop:steady} gives
$\|W_D-b\|_\infty\le C/D$ and $b/2\le W_D\le b$.
Set $\Phi=1/\phi$ on this positive interval. The mass constraint gives
\begin{equation}\label{eq:normalized-profile}
 U_D=\rho\,\frac{\Phi(W_D)}{\langle\Phi(W_D)\rangle}.
\end{equation}
The denominator is bounded away from zero uniformly, so
$\|U_D-\rho\|_\infty\le C/D$. Consequently
$\|U_DW_D-\rho b\|_\infty\le C/D$. The function
$R_D=W_D-b+\rho b\tau/D$ has zero boundary values and satisfies
\[
 \Delta R_D=\frac{U_DW_D-\rho b}{D}.
\]
Comparison with a constant multiple of $D^{-2}\tau$ proves
$\|R_D\|_\infty\le C/D^2$.

Put $h_D=W_D-b$. Since $\|h_D\|_\infty\le C/D$, Taylor's formula
gives the two expansions
\begin{align*}
 \Phi(W_D)&=\Phi(b)+\Phi'(b)h_D+O_{L^\infty}(D^{-2}),\\
 \langle\Phi(W_D)\rangle
          &=\Phi(b)+\Phi'(b)\langle h_D\rangle+O(D^{-2}).
\end{align*}
The remainders are uniform because $\Phi''$ is bounded on $[b/2,b]$.
The second line, with $\Phi(b)>0$, gives
\[
 \frac{1}{\langle\Phi(W_D)\rangle}
   =\frac{1}{\Phi(b)}
      -\frac{\Phi'(b)}{\Phi(b)^2}\langle h_D\rangle+O(D^{-2}).
\]
Multiplying the numerator and reciprocal denominator in
\eqref{eq:normalized-profile}, we obtain
\[
 U_D=\rho+\rho\frac{\Phi'(b)}{\Phi(b)}
                    (h_D-\langle h_D\rangle)+O_{L^\infty}(D^{-2}).
\]
The product of the two first-order terms is of order $D^{-2}$ and
is included in the remainder.
Substitute $h_D=-\rho b\tau/D+R_D$ and use
$\Phi'(b)/\Phi(b)=-\phi'(b)/\phi(b)$ to obtain
\eqref{eq:large-D-profiles}.
\end{proof}

\section{The nonlinear potential energy}\label{sec:energy}
Fix a mass $m$, assume $D>D_3(m)$, and abbreviate its steady state by
$(U,W)$. Constants below may depend on $D$. For an equal-mass
perturbation define
\begin{equation}\label{eq:potential}
 -\Delta q=u-U,\qquad \partial_\nu q=0,\qquad \int_\Omega q=0,
 \qquad \psi=w-W.
\end{equation}
Write $\Pi_0f=f-\langle f\rangle$, $a=\phi(W+\psi)$, and
$\Theta=\phi(W+\psi)-\phi(W)$. Since $U\phi(W)=Z$ is constant,
\begin{equation}\label{eq:potential-system}
\begin{aligned}
 q_t&=\Pi_0[a\Delta q-U\Theta],\\
 \psi_t&=D\Delta\psi-U\psi+(W+\psi)\Delta q.
\end{aligned}
\end{equation}
The homogeneous conditions are $\partial_\nu q=0$, $\psi=0$, and
$\int q=0$.

Put
$\mathcal J=u\phi(w)-U\phi(W)=-a\Delta q+U\Theta$.
The equations for the population and its potential imply
$-\Delta q_t=\Delta\mathcal J$. The boundary conditions make
$q_t+\mathcal J$ a Neumann harmonic function. Its spatial mean is
$\langle\mathcal J\rangle$, since $\int q_t=0$, and hence
\[
 q_t=-\mathcal J+\langle\mathcal J\rangle.
\]
\subsection{Low-order coercivity}
Define
\[
 E_0=\|\nabla q\|_2^2+\|\psi\|_2^2,\qquad
 D_0=\|\Delta q\|_2^2+\|\nabla\psi\|_2^2.
\]
Testing \eqref{eq:potential-system} with $-\Delta q$ and $\psi$ gives
the separate identities
\begin{align*}
 \frac12\frac{\dd}{\dd t}\|\nabla q\|_2^2
 +\int a|\Delta q|^2&=\int U\Theta\Delta q,\\
 \frac12\frac{\dd}{\dd t}\|\psi\|_2^2
 +D\|\nabla\psi\|_2^2+\int U\psi^2
 &=\int(W+\psi)\psi\Delta q.
\end{align*}
The first is the Neumann-inverse energy identity used in
\cite[Lemma~3.1]{TaoWinkler2023}, here centered at $(U,W)$.
No spatial derivative of $a$ occurs in its dissipative term.
Adding the identities yields
\begin{equation}\label{eq:low-identity}
 \frac12E_0'+\int_\Omega a|\Delta q|^2
 +D\|\nabla\psi\|_2^2+\int_\Omega U\psi^2
 =\int_\Omega[U\Theta+(W+\psi)\psi]\Delta q.
\end{equation}
The spatial mean in the first equation vanishes in this test because
$\int\Delta q=0$. On the range $W+\psi\in I_b$, the right-hand side is
bounded by $(M(m)L_b+5b/4)\|\psi\|_2\|\Delta q\|_2$.
Young's inequality and the Dirichlet Poincar\'e inequality yield
the bound
\begin{align*}
 &(M(m)L_b+5b/4)\|\psi\|_2\|\Delta q\|_2\\
 &\qquad\le\frac{a_b}{2}\|\Delta q\|_2^2
 +\frac{(M(m)L_b+5b/4)^2}{2a_b}\|\psi\|_2^2\\
 &\qquad\le\frac{a_b}{2}\|\Delta q\|_2^2
 +D_{\rm en}(m)\|\nabla\psi\|_2^2.
\end{align*}
We obtain
\begin{equation}\label{eq:low-energy}
 \frac12E_0'+\frac{a_b}{2}\|\Delta q\|_2^2
 +(D-D_{\rm en}(m))\|\nabla\psi\|_2^2\le0.
\end{equation}
We have $E_0'+cD_0\le0$ for some $c>0$ in any dimension. The
mean-zero Neumann and Dirichlet Poincar\'e inequalities give
$E_0\le C D_0$, and hence exponential decay while the nutrient remains
in $I_b$. Closing the argument requires a bound on $\|\psi\|_\infty$.
Since $E_0$ controls $q$ in $H^1$ and $\psi$ in $L^2$, we next estimate
one time derivative and recover the spatial regularity needed for
this bound.

\subsection{Energy equivalence and spatial recovery}
Let
\begin{equation}\label{eq:Y}
 \mathcal Y=\{(q,\psi)\in H^3(\Omega)^2:
       \partial_\nu q=0,\ \psi|_{\partial\Omega}=0,\ \int_\Omega q=0\}.
\end{equation}
The norm on $\mathcal Y$ is the product $H^3$ norm. For a state
$Q=(q,\psi)\in\mathcal Y$, write $Q_t=\mathcal F(Q)$ for the pair
defined by the right-hand sides of \eqref{eq:potential-system}.
It belongs to $H^1(\Omega)^2$ by Sobolev multiplication.
Along solutions it is the time derivative; at initial data this notation
means the derivative generated by the equation. A state is compatible if
\begin{equation}\label{eq:potential-compatibility}
 \psi_t|_{\partial\Omega}=0,
 \quad\hbox{equivalently}\quad
 (D\Delta\psi+b\Delta q)|_{\partial\Omega}=0.
\end{equation}
The trace in \eqref{eq:potential-compatibility} belongs to $H^{1/2}$;
the condition is linear in $Q$ and independent of the equilibrium mass.
Put $H_N^2=\{f\in H^2:\partial_\nu f=0\}$ and set
\begin{equation}\label{eq:energy-dissipation}
\begin{aligned}
 \mathcal E&=E_0+\|\nabla q\|_2^2+\|\nabla\psi\|_2^2
                       +\|\nabla q_t\|_2^2+\|\nabla\psi_t\|_2^2,\\
 \mathcal D&=D_0+\|\Delta q\|_2^2+\|\Delta\psi\|_2^2
                       +\|\Delta q_t\|_2^2+\|\Delta\psi_t\|_2^2.
\end{aligned}
\end{equation}
The low-order terms are added to the gradient energies of the state
and its first time derivative, giving weight two to
$\|\nabla q\|_2^2$ in $\mathcal E$ and to $\|\Delta q\|_2^2$ in
$\mathcal D$.
The dissipation is used only when $q_t\in H_N^2$ and
$\psi_t\in H^2\cap H_0^1$, as supplied at positive times by the
construction in Section~\ref{sec:global}.

The energy uses one time derivative, whose spatial regularity is two
orders lower than that of the state. Its initial value is determined
by $Q_t(0)=\mathcal F(Q_0)$. Compatibility supplies the zero Dirichlet trace of
its nutrient component. The zero-mean condition controls the population
component in $H^1$ without a normal trace. At positive times, parabolic
regularity supplies the $H^2$ boundary conditions needed in the
dissipation and in the differentiated energy tests.

\begin{lemma}[Elliptic recovery]\label{lem:recovery}
On a sufficiently small compatible $\mathcal Y$-ball, with $W+\psi\in I_b$,
there are uniform constants $c,C>0$ such that
\begin{equation}\label{eq:energy-equivalence}
 c\|Q\|_{H^3}^2\le\mathcal E\le C\|Q\|_{H^3}^2,
 \qquad \|Q_t\|_{H^1}\le C\mathcal E^{1/2}.
\end{equation}
If $q_t\in H_N^2$ and $\psi_t\in H^2\cap H_0^1$, then $Q\in H^4$ and
\begin{equation}\label{eq:dissipation-recovery}
 \|Q\|_{H^4}+\|Q_t\|_{H^2}\le C\mathcal D^{1/2},
 \qquad \mathcal E\le C_{\rm P}\mathcal D.
\end{equation}
Here $C_{\rm P}$ depends only on the domain.
\end{lemma}

\begin{proof}
Fix a compatible ball of radius $\rho\le1$ on which $W+\psi\in I_b$.
Constants may depend on the domain, $D$, the steady coefficient bounds,
$a_b$, and the $C^5(I_b)$ norm of $\phi$, but are independent of the
particular state in this ball.

For $n=2,3$, the embeddings $H^3\hookrightarrow W^{1,\infty}$ and
$H^2\hookrightarrow L^\infty$ give
\begin{equation}\label{eq:products}
 \|fg\|_{H^r}\le C\|f\|_{H^3}\|g\|_{H^r}\quad(0\le r\le3),
 \qquad \|fg\|_2\le C\|f\|_{H^1}\|g\|_{H^1}.
\end{equation}
For the second inequality, H\"older gives
$\|fg\|_2\le\|f\|_4\|g\|_4$, and $H^1\hookrightarrow L^4$
in both dimensions. Composition on the positive nutrient range yields
\begin{equation}\label{eq:coefficient-H3}
 \|a\|_{H^3}+\|a^{-1}\|_{H^3}
 +\|\phi'(W+\psi)\|_{H^3}+\|\phi''(W+\psi)\|_{H^3}\le C.
\end{equation}
The derivatives of $a$ contain derivatives of $W+\psi$ of total order
at most three. The analogous expansion of $\phi''(W+\psi)$ uses at
most five derivatives of $\phi$. Since $a\ge a_b$, applying the same
composition rule to its reciprocal gives the second term in
\eqref{eq:coefficient-H3}.

The coefficient difference vanishes at $\psi=0$. Factoring it as
\[
 \Theta=\psi\int_0^1\phi'(W+\theta\psi)\dd\theta
\]
and using the uniform multiplier bound therefore gives
\begin{equation}\label{eq:Theta-tame}
 \|\Theta\|_{H^r}\le C\|\psi\|_{H^r},\qquad r=0,1,2.
\end{equation}
The coefficient constants depend on the fixed $H^3$ radius.
Using these bounds in \eqref{eq:potential-system}, we obtain
\begin{align*}
 \|q_t\|_{H^1}
 &\le C\bigl(\|a\Delta q\|_{H^1}+\|U\Theta\|_{H^1}\bigr)\\
 &\le C\bigl(\|q\|_{H^3}+\|\psi\|_{H^1}\bigr),\\
 \|\psi_t\|_{H^1}
 &\le D\|\Delta\psi\|_{H^1}+C\|\psi\|_{H^1}
       +\|(W+\psi)\Delta q\|_{H^1}\\
 &\le C\|Q\|_{H^3}.
\end{align*}
Here $\Pi_0$ is bounded on $H^1$. Adding the lower-order terms in
$\mathcal E$ proves its upper bound in
\eqref{eq:energy-equivalence}.

For the reverse inequality, solve the first equation for $\Delta q$,
retaining its mean. Put $f=a\Delta q-U\Theta$ and $c(t)=\langle f\rangle$.
The Neumann condition gives
\[
 c(t)=-\frac1{|\Omega|}\int_\Omega\nabla a\cdot\nabla q
      -\frac1{|\Omega|}\int_\Omega U\Theta.
\]
Consequently
\begin{equation}\label{eq:mean-bound}
\begin{aligned}
 |c(t)|&\le C\bigl(\|\nabla q\|_2+\|\psi\|_2\bigr),\\
 \Delta q&=a^{-1}\bigl(q_t+U\Theta+c(t)\bigr).
\end{aligned}
\end{equation}
The mean-zero conditions for $q,q_t$ and the Dirichlet conditions for
$\psi,\psi_t$ imply
\begin{equation}\label{eq:energy-H1-control}
 \|Q\|_{H^1}+\|Q_t\|_{H^1}\le C\mathcal E^{1/2}.
\end{equation}
Using \eqref{eq:mean-bound} and \eqref{eq:Theta-tame} at order one,
\begin{align*}
 \|\Delta q\|_{H^1}
 &\le C\bigl(\|q_t\|_{H^1}+\|\Theta\|_{H^1}+|c(t)|\bigr)\\
 &\le C\bigl(\|Q_t\|_{H^1}+\|Q\|_{H^1}\bigr)
 \le C\mathcal E^{1/2}.
\end{align*}
The normalized Neumann estimate yields
\begin{equation}\label{eq:recover-q-H3}
 \|q\|_{H^3}\le C\bigl(\|\Delta q\|_{H^1}+\|q\|_2\bigr)
 \le C\mathcal E^{1/2}.
\end{equation}
We next rewrite the nutrient equation as
\begin{equation}\label{eq:nutrient-recovery}
 D\Delta\psi=\psi_t+U\psi-(W+\psi)\Delta q.
\end{equation}
The already established estimate for $\Delta q$ now gives
\begin{align*}
 D\|\Delta\psi\|_{H^1}
 &\le \|\psi_t\|_{H^1}+C\|\psi\|_{H^1}
       +C\|\Delta q\|_{H^1}\\
 &\le C\mathcal E^{1/2}.
\end{align*}
The Dirichlet estimate gives
$\|\psi\|_{H^3}\le C\mathcal E^{1/2}$,
which, together with \eqref{eq:recover-q-H3}, yields the lower bound in
\eqref{eq:energy-equivalence}.

To recover the dissipation norm, assume $q_t\in H_N^2$ and
$\psi_t\in H^2\cap H_0^1$. The Laplacian terms in $\mathcal D$ and
the scalar elliptic estimates give
\begin{equation}\label{eq:dissipation-H2-control}
 \|Q\|_{H^2}+\|Q_t\|_{H^2}\le C\mathcal D^{1/2}.
\end{equation}
With the coefficient bounds already obtained in $H^3$, the
distributional identity \eqref{eq:mean-bound} gives
\begin{align*}
 \|\Delta q\|_{H^2}
 &\le C\bigl(\|q_t\|_{H^2}+\|\Theta\|_{H^2}+|c(t)|\bigr)\\
 &\le C\bigl(\|Q_t\|_{H^2}+\|Q\|_{H^2}\bigr)
 \le C\mathcal D^{1/2}.
\end{align*}
Multiplication by $a^{-1}\in H^3$ is bounded on $H^2$, so this
calculation proves $\Delta q\in H^2$ before the Neumann estimate is
used to conclude
\[
 \|q\|_{H^4}\le C\bigl(\|\Delta q\|_{H^2}+\|q\|_2\bigr)
 \le C\mathcal D^{1/2}.
\]
Equation \eqref{eq:nutrient-recovery} then gives, in the same order,
\[
\begin{aligned}
 D\|\Delta\psi\|_{H^2}
 &\le\|\psi_t\|_{H^2}+C\|\psi\|_{H^2}
       +C\|\Delta q\|_{H^2}\\
 &\le C\mathcal D^{1/2},
\end{aligned}
\]
and hence $\|\psi\|_{H^4}\le C\mathcal D^{1/2}$.
The scalar estimates at orders one and two are valid on the $C^4$
domain \cite{AgmonDouglisNirenberg,Brezis2011}.

Finally, the inequality $\mathcal E\le C_{\rm P}\mathcal D$ follows
directly from the boundary conditions. For a mean-zero Neumann function
or a Dirichlet function $f\in H^2$, integration by parts and Poincar\'e
give
\[
 \|\nabla f\|_2^2=-\int f\Delta f
 \le C\|\nabla f\|_2\|\Delta f\|_2.
\]
Apply this to the four components of $Q,Q_t$, and use
$\|\psi\|_2\le C\|\nabla\psi\|_2$ for the remaining term in $E_0$.
The resulting constant $C_{\rm P}$ depends only on the domain.
\end{proof}

\subsection{The high-order energy inequality}\label{sec:high-energy}
\begin{proposition}\label{prop:energy}
There exist $\rho,c_{\rm en},C_{\rm rem}>0$ such that every solution
in \eqref{eq:strong-class}, while $\|Q\|_{H^3}<\rho$ and
$W+\psi\in I_b$, satisfies
\begin{equation}\label{eq:high-energy}
 \mathcal E'+c_{\rm en}\mathcal D
 \le C_{\rm rem}\mathcal E^{1/2}\mathcal D.
\end{equation}
\end{proposition}

\begin{proof}
Choose the radius so that $\|\psi\|_\infty\le b/4$.
The stationary bound $b/2\le W\le b$ then keeps $w=W+\psi$ in $I_b$.

\emph{The state equation.}
Test the two equations in \eqref{eq:potential-system} with
$-\Delta q$ and $-\Delta\psi$, respectively. The projected mean
vanishes because $\int_\Omega\Delta q=0$. Adding the identities gives
\begin{equation}\label{eq:high-identity}
\begin{aligned}
 &\frac12\frac{\dd}{\dd t}
       \bigl(\|\nabla q\|_2^2+\|\nabla\psi\|_2^2\bigr)
       +\int_\Omega a|\Delta q|^2+D\|\Delta\psi\|_2^2\\
 &\qquad=\int_\Omega U\Theta\Delta q
          +\int_\Omega U\psi\Delta\psi
          -\int_\Omega w\Delta q\Delta\psi.
\end{aligned}
\end{equation}
For $f\in H^2\cap H_0^1$, the Dirichlet eigenvalue inequality gives
$\|f\|_2\le\lambda_1^{-1}\|\Delta f\|_2$.
Together with $|\Theta|\le L_b|\psi|$ and
$\|U\|_\infty\le M(m)$, it yields
\begin{align*}
 \left|\int_\Omega U\Theta\Delta q\right|
 &\le M(m)L_b\|\psi\|_2\|\Delta q\|_2\\
 &\le \frac{M(m)L_b}{\lambda_1}
                  \|\Delta\psi\|_2\|\Delta q\|_2\\
 &\le\frac{a_b}{8}\|\Delta q\|_2^2
       +\frac{2}{a_b}\left(\frac{M(m)L_b}{\lambda_1}\right)^2
                         \|\Delta\psi\|_2^2.
\end{align*}
The other two terms are bounded by
\begin{align*}
 \left|\int_\Omega U\psi\Delta\psi\right|
 &\le\frac{M(m)}{\lambda_1}\|\Delta\psi\|_2^2,\\
 \left|\int_\Omega w\Delta q\Delta\psi\right|
 &\le\frac{5b}{4}\|\Delta q\|_2\|\Delta\psi\|_2\\
 &\le\frac{a_b}{8}\|\Delta q\|_2^2
       +\frac{2}{a_b}\left(\frac{5b}{4}\right)^2\|\Delta\psi\|_2^2.
\end{align*}
The coefficients of $\|\Delta\psi\|_2^2$ sum to $K(m)$.
Since $D>4K(m)$, \eqref{eq:high-identity} implies
\begin{equation}\label{eq:state-gradient-energy}
 \frac12\frac{\dd}{\dd t}\|\nabla Q\|_2^2
       +\frac{3a_b}{4}\|\Delta q\|_2^2
       +\frac{3D}{4}\|\Delta\psi\|_2^2\le0.
\end{equation}
\emph{The time-differentiated equation.}
The equilibrium is independent of time, and
$a_t=\Theta_t=\phi'(w)\psi_t$. Hence
\begin{equation}\label{eq:time-levels}
\begin{aligned}
 q_{tt}&=\Pi_0\bigl[a\Delta q_t-U\phi'(w)\psi_t
                              +\phi'(w)\psi_t\Delta q\bigr],\\
 \psi_{tt}&=D\Delta\psi_t-U\psi_t+w\Delta q_t+\psi_t\Delta q.
\end{aligned}
\end{equation}
Test these equations with $-\Delta q_t$ and $-\Delta\psi_t$. We obtain
\begin{equation}\label{eq:time-gradient-identity}
\begin{aligned}
 &\frac12\frac{\dd}{\dd t}\|\nabla Q_t\|_2^2
              +\int_\Omega a|\Delta q_t|^2+D\|\Delta\psi_t\|_2^2\\
 &\quad=\int_\Omega U\phi'(w)\psi_t\Delta q_t
        +\int_\Omega U\psi_t\Delta\psi_t
        -\int_\Omega w\Delta q_t\Delta\psi_t\\
 &\qquad\quad-\int_\Omega\phi'(w)\psi_t\Delta q\,\Delta q_t
              -\int_\Omega\psi_t\Delta q\,\Delta\psi_t.
\end{aligned}
\end{equation}
The first three integrals are estimated as above. Here
$|\phi'(w)\psi_t|\le L_b|\psi_t|$, and the zero Dirichlet trace of
$\psi_t$ allows the same use of $\lambda_1$, leaving the coefficients
$3a_b/4$ and $3D/4$ in the two dissipative terms.

For the last two integrals, H\"older's inequality and
$H^1\hookrightarrow L^4$ give
\begin{align*}
 \|\psi_t\Delta q\|_2
 &\le\|\psi_t\|_4\|\Delta q\|_4\\
 &\le C\|\psi_t\|_{H^1}\|\Delta q\|_{H^1}\\
 &\le C\|Q_t\|_{H^1}\|Q\|_{H^3}
 \le C\mathcal E.
\end{align*}
By Lemma~\ref{lem:recovery}, pairing with the Laplacians in
\eqref{eq:time-gradient-identity} gives
\begin{align*}
 &\left|\int_\Omega\phi'(w)\psi_t\Delta q\,\Delta q_t\right|
       +\left|\int_\Omega\psi_t\Delta q\,\Delta\psi_t\right|\\*
 &\qquad\le C\|\psi_t\Delta q\|_2
                   \bigl(\|\Delta q_t\|_2+\|\Delta\psi_t\|_2\bigr)\\*
 &\qquad\le C\mathcal E\mathcal D^{1/2}
 \le C\mathcal E^{1/2}\mathcal D.
\end{align*}
In the final inequality we used $\mathcal E\le C_{\rm P}\mathcal D$.
Consequently,
\begin{equation}\label{eq:time-gradient-energy}
 \frac12\frac{\dd}{\dd t}\|\nabla Q_t\|_2^2
       +\frac{3a_b}{4}\|\Delta q_t\|_2^2
       +\frac{3D}{4}\|\Delta\psi_t\|_2^2
 \le C\mathcal E^{1/2}\mathcal D.
\end{equation}

Adding \eqref{eq:low-energy}, \eqref{eq:state-gradient-energy}, and
\eqref{eq:time-gradient-energy} proves \eqref{eq:high-energy}.
For example, one may take
\[
 c_{\rm en}=2\min\left\{\frac{a_b}{2},
             D-D_{\rm en}(m),\frac{3a_b}{4},\frac{3D}{4}\right\}>0
\]
and increase $C_{\rm rem}$ accordingly, with both constants independent
of the time interval.

The regularity in \eqref{eq:strong-class} gives
$a\Delta q_t\in L^2L^2$. For the other differentiated product,
\[
 \|\psi_t\Delta q\|_{L^2L^2}
 \le C\|\psi_t\|_{L^2H^1}\|q\|_{C H^3}.
\]
Time difference quotients therefore give \eqref{eq:time-levels}.
For each $f\in\{q,\psi,q_t,\psi_t\}$, the solution class gives
\[
 f\in H^1(0,T;L^2)\cap L^2(0,T;H^2)
\]
with the Neumann or Dirichlet condition at almost every time.
The corresponding energy identity is
\[
 (f_t,-\Delta f)_{L^2}
       =\frac12\frac{\dd}{\dd t}\|\nabla f\|_2^2.
\]
The identity holds with initial traces in $H^1$ for the Neumann
component and $H_0^1$ for the Dirichlet component.
\end{proof}

\section{Global solutions and the selected equilibrium}\label{sec:global}
We construct solutions in the energy space and then use
\eqref{eq:high-energy} to continue them globally.

\begin{lemma}[Construction in the compatible energy space]\label{lem:local}
For a fixed positive steady state there are $T_{\rm loc},\rho_{\rm loc}>0$
such that every compatible $Q_0\in\mathcal Y$ with
$\|Q_0\|_{H^3}<\rho_{\rm loc}$ generates a unique solution of
\eqref{eq:potential-system} on $[0,T_{\rm loc}]$ in
\eqref{eq:strong-class}. The solution map is locally Lipschitz in these
norms. Compatibility is preserved, and a maximal solution satisfies
\begin{equation}\label{eq:continuation}
 \sup_{t<T_{\max}}\|Q(t)\|_{H^3}<\rho_{\rm loc}/2
 \quad\Longrightarrow\quad T_{\max}=\infty.
\end{equation}
The constants are uniform on compact families of steady states with a
common positive nutrient bound and the coefficient bounds of
Proposition~\ref{prop:steady}.
\end{lemma}

Appendix~\ref{app:construction} proves the lemma by solving two scalar
parabolic problems for the first time derivatives and recovering the
state by elliptic estimates. The common local existence time allows
a restart near any finite endpoint as long as the state stays below
$\rho_{\rm loc}/2$.

\subsection{Fixed-mass stability}
Choose $C_0\ge1$ so that the energy equivalence reads
\[
 C_0^{-1}\|Q\|_{H^3}^2\le\mathcal E\le C_0\|Q\|_{H^3}^2
\]
on the ball in Proposition~\ref{prop:energy}.
Choose an energy level $\eta>0$ such that
\begin{equation}\label{eq:small-radius}
 \sqrt{C_0}\eta<\min\{\rho,\rho_{\rm loc}/2\},
 \qquad C_{\rm rem}\eta<\frac{c_{\rm en}}2.
\end{equation}
Then decrease $\eta$, if necessary, so that Sobolev embedding gives
$\|\psi\|_\infty<b/4$ whenever $\|Q\|_{H^3}\le\sqrt{C_0}\eta$.
Choose the data radius $\delta_3>0$ with
\[
 C_0\delta_3^2<\frac{\eta^2}{4}.
\]

For compatible $\|Q_0\|_{H^3}<\delta_3$, Lemma~\ref{lem:local}
provides a solution. Its energy is continuous, since
$Q\in C H^3$ and $Q_t\in C H^1$, and
$\mathcal E(0)<\eta^2/4$. Consider the maximal time interval on
which the state remains in the energy ball and
$\mathcal E(t)^{1/2}<\eta$. On this interval,
\[
 C_{\rm rem}\mathcal E^{1/2}\mathcal D
       \le\frac{c_{\rm en}}2\mathcal D.
\]
The high-order estimate therefore yields
\begin{equation}\label{eq:absorbed}
\begin{aligned}
 \mathcal E'+\frac{c_{\rm en}}2\mathcal D&\le0,\\
 \mathcal E(t)+\frac{c_{\rm en}}2\int_0^t\mathcal D(s)\dd s
                                      &\le\mathcal E(0).
\end{aligned}
\end{equation}
The second inequality follows by integration first from a positive
time and then by passage to zero, using continuity of the energy.
It gives $\mathcal E(t)^{1/2}<\eta/2$. The lower energy bound now
keeps $\|Q(t)\|_{H^3}<\sqrt{C_0}\eta/2$, strictly inside both the
energy ball and the continuation radius. Continuity therefore extends
the estimate to the full maximal interval, and
\eqref{eq:continuation} gives global existence.

Since $\mathcal E\le C_{\rm P}\mathcal D$, the first inequality in
\eqref{eq:absorbed} also gives
\[
 \mathcal E'+\frac{c_{\rm en}}{2C_{\rm P}}\mathcal E\le0,
 \qquad
 \mathcal E(t)\le\mathcal E(0)
                e^{-c_{\rm en}t/(2C_{\rm P})}.
\]
Taking square roots and using the energy equivalence on both sides,
we obtain
\begin{equation}\label{eq:fixed-decay}
 \|Q(t)\|_{H^3}\le C_0e^{-c_{\rm en}t/(4C_{\rm P})}
                         \|Q_0\|_{H^3}.
\end{equation}
The integrated estimate in \eqref{eq:absorbed}, together with elliptic
recovery at the dissipation level, gives
\begin{equation}\label{eq:global-dissipation-bound}
 \int_0^\infty
 \bigl(\|Q(t)\|_{H^4}^2+\|Q_t(t)\|_{H^2}^2\bigr)\dd t
 \le C\|Q_0\|_{H^3}^2.
\end{equation}

\subsection{Recovery of the physical solution}

Reconstruct $u=U-\Delta q$, $w=W+\psi$. If
$\mathcal J=-a\Delta q+U\Theta=u\phi(w)-Z$, then
$q_t=-\mathcal J+\langle\mathcal J\rangle$; applying $-\Delta$ recovers
the population equation, and the stationary nutrient identity recovers
the second equation in \eqref{eq:model}. At almost every positive time,
this calculation reads explicitly
\begin{align*}
 u_t&=-\Delta q_t=\Delta\mathcal J=\Delta(u\phi(w)),\\
 w_t&=D\Delta\psi-U\psi+w\Delta q\\
    &=D\Delta w-D\Delta W-U(w-W)+w(U-u)
     =D\Delta w-uw.
\end{align*}
The steady identity $D\Delta W=UW$ is used in the last equality.
Furthermore,
$\mathcal J\in H^2$ and $\partial_\nu\mathcal J=-\partial_\nu q_t=0$.
In particular $u_t\in L^2(0,T;L^2)$ and
$u\phi(w)\in L^2(0,T;H^2)$, with the zero-flux trace.
The Neumann condition gives mass conservation:
\[
 \int_\Omega(u-U)=-\int_\Omega\Delta q
 =-\int_{\partial\Omega}\partial_\nu q=0.
\]
The nutrient condition follows from $\psi=0$. For the population flux,
the identity $u\phi(w)=Z-q_t+\langle\mathcal J\rangle$ gives
\[
 \partial_\nu(u\phi(w))=0
 \quad\hbox{in }L^2(0,T;H^{1/2}(\partial\Omega)).
\]
Normalized Neumann regularity gives
\begin{equation}\label{eq:physical-norm}
 \|Q\|_{H^3}\asymp\|u-U\|_{H^1}+\|w-W\|_{H^3}.
\end{equation}
Conversely, a strong physical solution gives the same normalized potential
system, so uniqueness and local Lipschitz dependence transfer to these
variables. Finitely many local steps cover every finite time interval.

\subsection{Uniformity in mass and proof of Theorem~\ref{thm:main}}
The functions $D_{\rm pos}$, $D_{\rm en}$, and $K$ are continuous
in mass. Since $D>D_3(m)$, choose $0<r_m<m/2$ so that
\[
 D-D_3(\mu)\ge\frac{D-D_3(m)}2>0
       \qquad\hbox{for }\mu\in I_m=[m-r_m,m+r_m].
\]
Proposition~\ref{prop:steady} gives common positive nutrient bounds
and uniform $W^{3,\infty}$ coefficient bounds on this interval.
The constants in elliptic recovery, construction, and the energy
estimate may therefore be chosen uniformly. The choices of $\eta$
and $\delta_3$ in the preceding argument give a common data radius
and decay rate for every $\mu\in I_m$.

For data near $(U_m,W_m)$, Cauchy--Schwarz gives
$|\mu_0-m|\le|\Omega|^{1/2}\|u_0-U_m\|_2\le Cd_0$.
Choose $\delta$ to ensure $\mu_0\in I_m$. The steady Lipschitz estimates
and the normalized Neumann inverse at mass $\mu_0$ then give
\eqref{eq:recentering}. The inverse is applied to $u_0-U_{\mu_0}$,
which has zero mean. More explicitly, the profile estimates yield
\begin{align*}
 \|u_0-U_{\mu_0}\|_{H^1}
 &\le\|u_0-U_m\|_{H^1}+C|\mu_0-m|,\\
 \|w_0-W_{\mu_0}\|_{H^3}
 &\le\|w_0-W_m\|_{H^3}+C|\mu_0-m|.
\end{align*}
Applying the normalized Neumann inverse to the first difference gives
$\|Q_0\|_{H^3}\le Cd_0$. The constant is uniform because the domain
is fixed and the profile estimates are uniform on $I_m$.
The nutrient compatibility condition also survives the change of center:
\[
 (D\Delta\psi_0+b\Delta q_0)|_{\partial\Omega}
 =(D\Delta w_0-bu_0)|_{\partial\Omega},
\]
using $D\Delta W_{\mu_0}=U_{\mu_0}W_{\mu_0}$ and
$W_{\mu_0}|_{\partial\Omega}=b$.
After reducing $\delta$ so that $\|Q_0\|_{H^3}<\delta_3$ uniformly,
the fixed-mass result proves \eqref{eq:physical-decay} and
\eqref{eq:potential-decay}. Adding the steady profile difference proves
\eqref{eq:mass-selection-reference-offset}. The mass functional is
Lipschitz in $L^2$, so Proposition~\ref{prop:steady} also proves the
claimed Lipschitz dependence of the selected equilibrium.

If the masses differ, convergence to $(U_m,W_m)$ is impossible even
in $L^2$ for the population, since
\[
 |\mu_0-m|=\left|\int_\Omega(u(t)-U_m)\right|
 \le |\Omega|^{1/2}\|u(t)-U_m\|_2.
\]
The offset in \eqref{eq:mass-selection-reference-offset} is therefore
forced by conservation.

\subsection{Dependence on the data and nonnegativity}
We now compare solutions whose initial data have different nearby masses. Let $Q_i=(q_i,\psi_i)$ be the potential
solution centered at $(U_i,W_i)=(U_{\mu_i},W_{\mu_i})$, $i=1,2$.
Set $p=q_1-q_2$, $\chi=\psi_1-\psi_2$, and write
$w_i=W_i+\psi_i$, $a_i=\phi(w_i)$,
$\Theta_i=\phi(w_i)-\phi(W_i)$. Subtracting the equations gives
\begin{equation}\label{eq:difference-system}
\begin{aligned}
 p_t={}&\Pi_0\bigl[a_1\Delta p-U_1(\Theta_1-\Theta_2)\bigr]\\
       &+\Pi_0\bigl[(a_1-a_2)\Delta q_2-(U_1-U_2)\Theta_2\bigr],\\
 \chi_t={}&D\Delta\chi-U_1\chi+w_1\Delta p\\
       &-(U_1-U_2)\psi_2+(w_1-w_2)\Delta q_2.
\end{aligned}
\end{equation}
Both differences satisfy the same homogeneous boundary conditions as
the potentials, and $\int p=0$.

The steady-profile estimate gives
\[
 \|U_1-U_2\|_{W^{3,\infty}}
 +\|W_1-W_2\|_{W^{3,\infty}}\le C|\mu_1-\mu_2|.
\]
To keep the coefficient differences linear in these quantities, use
\[
 a_1-a_2=(w_1-w_2)\int_0^1
 \phi'(w_2+\theta(w_1-w_2))\dd\theta,
 \qquad w_1-w_2=\chi+W_1-W_2.
\]
For the other coefficient difference, factoring out the perturbation
before subtraction gives
\begin{align*}
 \Theta_1-\Theta_2
 ={}&\chi\int_0^1\phi'(W_1+\theta\psi_1)\dd\theta\\
    &+\psi_2\int_0^1
       [\phi'(W_1+\theta\psi_1)-\phi'(W_2+\theta\psi_2)]\dd\theta.
\end{align*}
In the second integral, the argument difference is
$W_1-W_2+\theta\chi$. The first term is bounded by the mass
difference; the second contributes to the estimate for $\chi$.
For $r=1,2$, the product and composition estimates therefore give
\begin{align*}
 \|a_1-a_2\|_{H^r}
 &\le C\bigl(\|\chi\|_{H^r}+|\mu_1-\mu_2|\bigr),\\
 \|\Theta_1-\Theta_2\|_{H^r}
 &\le C\|\chi\|_{H^r}
       +C|\mu_1-\mu_2|\|\psi_2\|_{H^r}.
\end{align*}
The constants are common to the small solution ball. In deriving the
second estimate, terms containing $\chi$ and derivatives of $\psi_2$
are bounded by the same product estimates as in
Lemma~\ref{lem:recovery}. Differentiation in time uses
$\partial_tW_i=\partial_tU_i=0$ and the bounds for $Q_{i,t}$.

Let $\mathcal X_T$ denote the sum of the norms in
\eqref{eq:strong-class}, as in Appendix~\ref{app:construction}.
Write $\mathcal A_i,\mathcal N_i$ for the frozen operator and remainder
in \eqref{eq:frozen} at $(U_i,W_i)$, $i=1,2$.
On a common local interval, write both equations using $\mathcal A_1$. The additional linear forcing in the second
solution contains the steady coefficient differences multiplied by
$\Delta q_2$, $\psi_2$, or their time derivatives. Its norm is controlled by the mass difference as follows.
For example, let $c=\phi(W_1)-\phi(W_2)$ and
$F=\Pi_0[c\Delta q_2]$. The profile bounds imply
$\|c\|_{W^{2,\infty}}\le C|\mu_1-\mu_2|$, and $c$ is independent
of time. Thus
\begin{equation}\label{eq:mass-forcing-bound}
\begin{aligned}
 \|F\|_{L^2H^2}&\le C|\mu_1-\mu_2|\,\|q_2\|_{L^2H^4},\\
 \|F_t\|_{L^2L^2}&\le C|\mu_1-\mu_2|\,\|q_{2,t}\|_{L^2H^2},\\
 \|F(0)\|_{H^1}&\le C|\mu_1-\mu_2|\,\|q_2(0)\|_{H^3}.
\end{aligned}
\end{equation}
The $L^2L^2$ part of the forcing norm is controlled by the $L^2H^2$
estimate. The products $(W_1-W_2)\Delta q_2$ and $(U_1-U_2)\psi_2$
satisfy the same three bounds, with $\|Q_2\|_{\mathcal X_T}$ on the
right. For the nonlinear terms, insert $\mathcal N_1(Q_2)$ between
$\mathcal N_1(Q_1)$ and $\mathcal N_2(Q_2)$. The first difference is
bounded by \eqref{eq:local-nonlinear}; in the second, the coefficient
factorizations above give a factor $|\mu_1-\mu_2|$, and each term
vanishes at $Q_2=0$. Combining the bounds gives
\[
\begin{aligned}
 &\|\mathcal N_1(Q_1)-\mathcal N_2(Q_2)
                   +(\mathcal A_2-\mathcal A_1)Q_2\|_{\mathcal G_T}\\
 &\qquad\le C(R+\sqrt T)\|Q_1-Q_2\|_{\mathcal X_T}
               +C|\mu_1-\mu_2|\|Q_2\|_{\mathcal X_T},
\end{aligned}
\]
where $\|Q_i\|_{\mathcal X_T}\le R$.
The initial time derivatives satisfy
\[
 \|\mathcal F_{\mu_1}(Q_{1,0})-\mathcal F_{\mu_2}(Q_{2,0})\|_{H^1}
 \le C\bigl(\|Q_{1,0}-Q_{2,0}\|_{H^3}+|\mu_1-\mu_2|\bigr),
\]
where the subscript records the equilibrium used in the potential
equation. The linear estimate \eqref{eq:local-linear} and absorption
therefore give
\begin{equation}\label{eq:finite-time-dependence}
 \|Q_1-Q_2\|_{\mathcal X_T}
 \le C_T\bigl(\|Q_{1,0}-Q_{2,0}\|_{H^3}+|\mu_1-\mu_2|\bigr).
\end{equation}
Initially this holds on the common local interval. Repetition proves
it on every finite interval, with a constant allowed to depend on $T$.

For the physical initial data, the right-hand side is controlled by
$\|u_{1,0}-u_{2,0}\|_{H^1}+\|w_{1,0}-w_{2,0}\|_{H^3}$:
the mass functional is bounded on $L^2$, the profile map is Lipschitz,
and the difference of the two centered populations has zero mean.
Reconstruction then proves the finite-time Lipschitz assertion in
Theorem~\ref{thm:main}. For identical initial data, the mass difference
vanishes and the initial potentials agree, so \eqref{eq:finite-time-dependence}
also gives uniqueness.

Finally, nonnegativity follows by testing the weak population equation
with the negative part. With $a=\phi(w)\ge a_b$ and
$u^-=\max\{-u,0\}$,
\[
 \frac12\frac{\dd}{\dd t}\|u^-\|_2^2
 +\frac{a_b}{2}\|\nabla u^-\|_2^2
 \le\frac{\|\nabla a\|_\infty^2}{2a_b}\|u^-\|_2^2.
\]
The coefficient is bounded on finite intervals by $w\in C H^3$.
Time regularization justifies the test and Gronwall's inequality gives
$u^-=0$ if $u_0^-=0$.

\section{Numerical results}\label{sec:numerics}
The computations use $b=1$ on the ellipse
$\Omega=\{(x,y):x^2+(y/0.65)^2<1\}$.
We examine relaxation from compatible data and the dependence of the
stationary profile on motility and nutrient diffusion. Compatibility
and the diffusion condition are checked analytically for the continuous
initial family. The admissible perturbation radius remains implicit.
We track relaxation from finite perturbations using lumped $L^2$
errors and a discrete low-order energy. Appendix~\ref{app:numerics}
gives the method, fitted decay rates, and checks of amplitude, mesh
size, and solver tolerances.

\subsection{Compatible perturbations}
For $\phi(s)=1-e^{-s}$, we take $D=60$. The torsion function and an
enclosing-rectangle bound for the first Dirichlet eigenvalue give
\begin{equation}\label{eq:numerical-threshold}
 56.5101<D_3(\mu)<57.0865<60,
 \qquad 0.999\le\mu\le1.001.
\end{equation}
Appendix~\ref{app:compatible-data} derives these bounds analytically.
Let $(U_1,W_1)$ be the mass-$1$ equilibrium, set $X=x$, $Y=y/0.65$,
and let $\eta=\eta(X,Y)$ be a smooth radial bump supported in
$X^2+Y^2<0.75^2$. We use
\begin{equation}\label{eq:compatible-numerical-data}
\begin{aligned}
 u_0&=U_1+\varepsilon\eta\,(X+0.4Y)
               +(\mu-1)\frac{\eta}{\int_\Omega\eta},\\
 w_0&=W_1+\varepsilon\eta\,(1+0.3X-0.2Y).
\end{aligned}
\end{equation}
The odd term has zero integral, and both perturbations vanish near
the boundary. The data therefore have mass $\mu$ and satisfy
\eqref{eq:physical-compatibility}. Their $H^1\times H^3$ distance from
the reference state is $O(\varepsilon+|\mu-1|)$, so sufficiently small
members lie within the theorem. We use $\varepsilon=10^{-3}$ below.

For $\mu=0.999,1,1.001$, the computed solutions approach their
mass-selected equilibria. Unequal-mass trajectories retain a nonzero
distance from the mass-$1$ reference. Figures~\ref{fig:compatible-snapshots}
and \ref{fig:compatible-comparison} illustrate the convergence and
mass selection. Compared with compatible data at $D=5$, the $D=60$
solutions have a shorter initial nutrient transient, while the later
fitted energy rates are close. Decay is also observed at $D=5$, below
the sufficient diffusion bound.

\subsection{Motility and stationary shape}\label{sec:numerical-profiles}
To test Proposition~\ref{prop:large-D}, take mass $1$ and compare
\[
 \phi(s)=s\quad\hbox{and}\quad\phi(s)=\frac{2s}{1+s},
 \qquad D=60,120,240.
\]
Both laws have $\phi(1)=1$, while $\kappa_1=1$ and $1/2$,
respectively. Their sufficient threshold upper bounds are
$50.917$ and $31.860$, so all three diffusivities
are admissible. The nutrient deficit has the same leading profile;
the saturating law has half the leading population contrast.

Figure~\ref{fig:large-D-response}(a) compares the rescaled population
with the prediction. We center at $\rho_h=1/|\Omega_h|$ to account
for the polygonal area. Panel (b) uses discrete torsion coefficients
on 66049 nodes. The scaled remainders decrease approximately as
$D^{-1}$, corresponding to $D^{-2}$ before rescaling.
Appendix~\ref{app:expansion-numerics} checks coefficient convergence
separately to distinguish geometry error from the expansion remainder.

\begin{figure}[htbp]
 \centering
 \includegraphics[width=\textwidth]{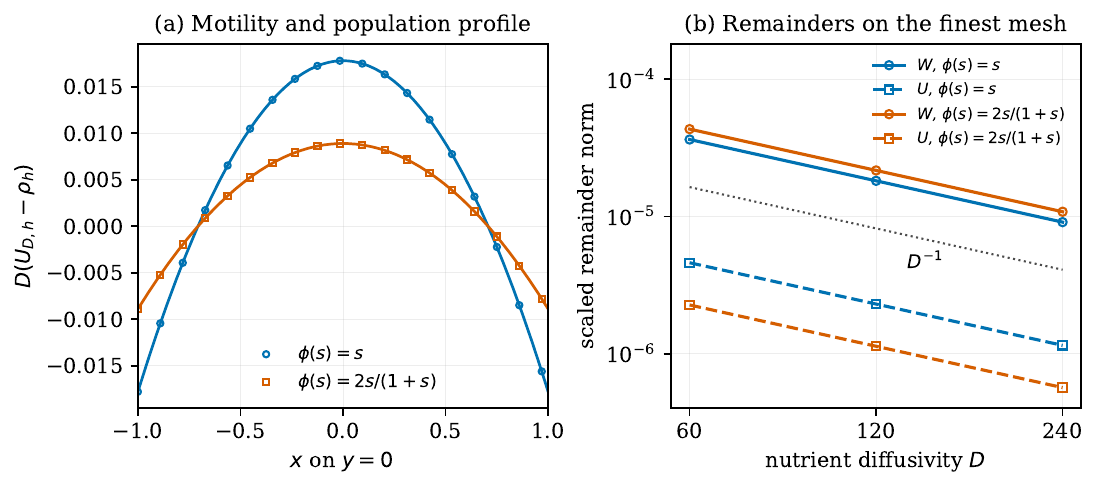}
 \caption{Large-diffusion profiles for two normalized motilities.
 (a) $D(U_{D,h}-\rho_h)$ at $D=240$ (markers) and its leading
 continuum profile (lines); they nearly coincide. The saturating law
 gives half the leading population response. (b) The scaled remainder
 norms $e_{W,h},e_{U,h}$ defined in Appendix~\ref{app:expansion-numerics}.
 The dotted line indicates order $D^{-1}$.}
 \label{fig:large-D-response}
\end{figure}

\FloatBarrier
\Needspace{7\baselineskip}
\appendix
\section{Regularity and construction details}\label{app:technical}

\subsection{Regularity of the steady profiles}\label{app:steady-proof}
\begin{proof}
We complete the proof of Proposition~\ref{prop:steady}. Let $\mu\ge\nu$
belong to the compact mass interval $J$, and write $r=W_\nu-W_\mu$.
The positive range $[b/2,b]$ and the bounds already proved in
Section~\ref{sec:steady} give uniform bounds for $g$, its needed
derivatives, and the divided difference in \eqref{eq:steady-difference}.
The Dirichlet estimate then gives
$\|r\|_{W^{2,p}}\le C_J|\mu-\nu|$.
For $p>n$, the algebra property of $W^{1,p}$ and the scalar elliptic
estimates through order four give uniform $W^{4,p}$ bounds for $W_\mu$.
Applying them to
\[
 D\Delta(W_\mu-W_\nu)
 =(Z_\mu-Z_\nu)g(W_\mu)+Z_\nu[g(W_\mu)-g(W_\nu)]
\]
successively at orders one and two gives
$\|W_\mu-W_\nu\|_{W^{4,p}}\le C_J|\mu-\nu|$.
Factoring the last difference as $(W_\mu-W_\nu)$ times an
averaged $g'$ bounds its $W^{r,p}$ norm by $C_J\|W_\mu-W_\nu\|_{W^{r,p}}$
for $r=1,2$. The scalar estimates hold on a $C^4$ domain
\cite{AgmonDouglisNirenberg,Brezis2011}.
Finally, $U_\mu=Z_\mu\Phi(W_\mu)$ with $\Phi=1/\phi$.
The embedding $W^{4,p}\hookrightarrow W^{3,\infty}$ and the $C^4$ bound
for $\Phi$ give \eqref{eq:mass-high-bounds} by differentiating the
difference through order three. On $I$, the same argument at orders zero
and one proves \eqref{eq:mass-interval-bounds} under $\phi\in C^3$.
\end{proof}

On $I=(0,1)$ the torsion function is $\tau=x(1-x)/2$, so
$D_{\rm pos}(\mu)=\mu/4$. Reflection symmetry and uniqueness give
$W_\mu'(1/2)=0$. Since $0\le W_\mu''\le b\|U_\mu\|_\infty/D$,
\begin{equation}\label{eq:interval-derivative}
 \|W_\mu'\|_\infty\le\frac{b\|U_\mu\|_\infty}{2D},
 \qquad \|U_\mu'\|_\infty\le\frac{C_J}{D}.
\end{equation}
Here and in the interval energy estimate, the coefficient bounds may be
chosen independently of large $D$ on a fixed compact mass interval.

\subsection{Proof of Lemma~\ref{lem:local}}\label{app:construction}
\begin{lemma}[Scalar estimates for the frozen system]\label{lem:scalar-frozen}
Let $0<a_0\le a_*\le a_1$, with $a_*\in W^{1,\infty}(\Omega)$,
let $W\in L^\infty(\Omega)$, and fix $D>0$. Suppose
$r_0\in H^1(\Omega)$ has zero mean, $s_0\in H_0^1(\Omega)$, and
$f,g\in L^2(0,T;L^2)$ with $\langle f(t)\rangle=0$.
For $0<T\le1$, the successive scalar problems
\begin{equation}\label{eq:scalar-frozen}
\begin{aligned}
 r_t-\Pi_0[a_*\Delta r]&=f,\\
 s_t-D\Delta s&=g+W\Delta r
\end{aligned}
\end{equation}
with $\partial_\nu r=0$, $\langle r\rangle=0$, $s|_{\partial\Omega}=0$,
and $(r,s)(0)=(r_0,s_0)$ have a unique solution satisfying
\begin{equation}\label{eq:scalar-frozen-bound}
 \|(r,s)\|_{C H^1}+\|(r,s)\|_{L^2H^2}
 +\|(r_t,s_t)\|_{L^2L^2}
 \le C\bigl(\|(r_0,s_0)\|_{H^1}
                  +\|(f,g)\|_{L^2L^2}\bigr).
\end{equation}
The $H^2$ domains are $H_N^2$ and $H^2\cap H_0^1$, respectively.
The constant is independent of $T\le1$ and is uniform for fixed
$\Omega,D$ and common bounds on $a_0^{-1},a_1,\|W\|_\infty$.
\end{lemma}

\begin{proof}
First solve $h_t-a_*\Delta h=f$, $\partial_\nu h=0$, $h(0)=r_0$.
The scalar realization of $-a_*\Delta$ has domain $H_N^2$ in
$L^2(\Omega,a_*^{-1}\dd x)$: it is the nonnegative operator associated
with the form $\int\nabla h\cdot\nabla v$ on $H^1$.
Its form domain, and hence the initial trace space, is $H^1$.
Scalar $L^2$ parabolic regularity therefore gives
$h\in C H^1\cap L^2H_N^2$ and $h_t\in L^2L^2$
\cite{DenkHieberPruss,Solonnikov1965}.
The weighted and unweighted $L^2$ norms are equivalent.
Set $r=h-\langle h\rangle$. Since $f$ has zero mean and
$\Delta r=\Delta h$, subtraction of the evolving mean gives the
first equation of \eqref{eq:scalar-frozen} with initial value $r_0\in H^1$.

The uniform estimate can be seen directly. Testing with $-\Delta r$
and using the mean-zero Neumann inequalities gives
\[
 \frac12\frac{\dd}{\dd t}\|\nabla r\|_2^2
 +\frac{a_0}{2}\|\Delta r\|_2^2
 \le\frac{1}{2a_0}\|f\|_2^2.
\]
The equation bounds $r_t$ in $L^2L^2$. Next solve the Dirichlet heat
equation for $s$ with source $g+W\Delta r\in L^2L^2$ and initial
trace $s_0\in H_0^1$. Testing it with $-\Delta s$ yields
\[
 \frac12\frac{\dd}{\dd t}\|\nabla s\|_2^2
 +\frac D2\|\Delta s\|_2^2
 \le\frac1D\bigl(\|g\|_2^2+\|W\|_\infty^2\|\Delta r\|_2^2\bigr).
\]
Scalar elliptic estimates, the two equations, and their $H^1$ trace
continuity prove \eqref{eq:scalar-frozen-bound}. The same estimates
applied to differences give uniqueness. They also prove the claimed
uniformity without a time-dependent constant.
\end{proof}

\begin{proof}[Proof of Lemma~\ref{lem:local}]
For $T\le1$, write
\[
 \mathcal G_T=H^1(0,T;L^2(\Omega)^2)
                      \cap L^2(0,T;H^2(\Omega)^2),
 \qquad
 \|G\|_{\mathcal G_T}
       =\|G_t\|_{L^2L^2}+\|G\|_{L^2H^2}.
\]
Let $\mathcal X_T$ be the space in \eqref{eq:strong-class}, with
the stated boundary and mean conditions and norm
\begin{equation}\label{eq:construction-norm}
\begin{aligned}
 \|Q\|_{\mathcal X_T}
   ={}&\|Q\|_{C H^3}+\|Q\|_{L^2H^4}
       +\|Q_t\|_{C H^1}\\
     &+\|Q_t\|_{L^2H^2}+\|Q_{tt}\|_{L^2L^2}.
\end{aligned}
\end{equation}
Freeze the principal population coefficient at $a_*=\phi(W)$.
Then the system takes the form
\begin{equation}\label{eq:frozen}
 Q_t+\mathcal A_*Q=\mathcal N(Q),
\end{equation}
where
\[
 \mathcal A_*Q=
 \begin{pmatrix}-\Pi_0[a_*\Delta q]\\-W\Delta q-D\Delta\psi\end{pmatrix},
 \qquad
 \mathcal N(Q)=
 \begin{pmatrix}
   \Pi_0[(a-a_*)\Delta q-U\Theta]\\
   \psi\Delta q-U\psi
 \end{pmatrix}.
\]
First replace $\mathcal N(Q)$ by $G\in\mathcal G_T$, with mean-zero
first component. Suppose that the derivative prescribed at time zero is
\[
 \dot Q_0=G(0)-\mathcal A_*Q_0\in H^1\times H_0^1.
\]
The dot here denotes this prescribed initial derivative.
We find $Q_t=(r,s)$ by solving successively
\begin{equation}\label{eq:frozen-derivatives}
\begin{aligned}
 r_t-\Pi_0[a_*\Delta r]&=(G_1)_t,\\
 s_t-D\Delta s&=(G_2)_t+W\Delta r,
\end{aligned}
\end{equation}
with $\partial_\nu r=0$, $\langle r\rangle=0$, $s=0$ on the boundary,
and $(r,s)(0)=\dot Q_0$. Lemma~\ref{lem:scalar-frozen} gives
\[
 \|(r,s)\|_{C H^1}+\|(r,s)\|_{L^2H^2}
       +\|(r_t,s_t)\|_{L^2L^2}
 \le C\bigl(\|\dot Q_0\|_{H^1}+\|G_t\|_{L^2L^2}\bigr).
\]
Define $Q(t)=Q_0+\int_0^t(r,s)(\sigma)\dd\sigma$.
The residual $Q_t+\mathcal A_*Q-G$ has zero time derivative and
vanishes at zero, so $Q$ solves the forced equation.

To recover the spatial derivatives, first note the trace estimate
\begin{equation}\label{eq:forcing-trace}
 \|G\|_{C H^1}^2
 \le C\|G(0)\|_{H^1}^2
       +C\|G_t\|_{L^2L^2}\|G\|_{L^2H^2}.
\end{equation}
For this estimate, extend $G$ spatially by a fixed extension operator
bounded on $L^2,H^1,H^2$. For the extension $\widetilde G$, the
identity
\[
 \frac12\frac{\dd}{\dd t}\|\widetilde G\|_{H^1(\R^n)}^2
     =(\widetilde G_t,(1-\Delta)\widetilde G)_{L^2(\R^n)}
\]
and Cauchy--Schwarz prove \eqref{eq:forcing-trace}, first for smooth
functions and then by density. The spatial extension makes the constant
independent of $T$ and allows unrestricted boundary values for $G$.

The population equation gives
\[
 a_*\Delta q=q_t-G_1+c(t),\qquad
 c(t)=-\frac{\int_\Omega a_*^{-1}(q_t-G_1)}
                  {\int_\Omega a_*^{-1}}.
\]
The mean correction enforces $\int_\Omega\Delta q=0$.
For $k=1,2$, the stationary coefficient bounds and the normalized
Neumann estimate give
\[
 \|q\|_{H^{k+2}}
       \le C\bigl(\|q_t\|_{H^k}+\|G_1\|_{H^k}\bigr).
\]
Having recovered $q$, use
$D\Delta\psi=\psi_t-G_2-W\Delta q$ and the Dirichlet estimate to obtain
\[
 \|\psi\|_{H^{k+2}}
       \le C\bigl(\|\psi_t\|_{H^k}+\|G_2\|_{H^k}
                                  +\|\Delta q\|_{H^k}\bigr).
\]
Apply the $k=1$ bounds in continuous time and the $k=2$ bounds in
$L^2$ time. Since $G(0)=\dot Q_0+\mathcal A_*Q_0$, the result is
\begin{equation}\label{eq:local-linear}
 \|Q\|_{\mathcal X_T}
 \le C\bigl(\|Q_0\|_{H^3}+\|\dot Q_0\|_{H^1}
                                  +\|G\|_{\mathcal G_T}\bigr).
\end{equation}
All constants are uniform for $T\le1$ and for the nearby steady
profiles under consideration.

We now estimate $\mathcal N(Q)$ on the ball
$\|Q\|_{\mathcal X_T}\le R$, with $R$ small enough to keep
$W+\psi\in I_b$. Since $a-a_*=\Theta$, the spatial estimate for
the principal nonlinear term is
\begin{align*}
 \|(a-a_*)\Delta q\|_{L^2H^2}
 &\le C\|\psi\|_{C H^3}\|\Delta q\|_{L^2H^2}\\
 &\le C\|\psi\|_{C H^3}\|q\|_{L^2H^4}\le CR^2.
\end{align*}
Its time derivative is
\[
 \partial_t[(a-a_*)\Delta q]
     =\phi'(W+\psi)\psi_t\Delta q+(a-a_*)\Delta q_t.
\]
Estimate the two terms separately:
\begin{align*}
 \|\phi'(W+\psi)\psi_t\Delta q\|_{L^2L^2}
 &\le C\|\psi_t\|_{L^2H^1}\|q\|_{C H^3}\le CR^2,\\
 \|(a-a_*)\Delta q_t\|_{L^2L^2}
 &\le C\|\psi\|_{C H^3}\|q_t\|_{L^2H^2}\le CR^2.
\end{align*}
Here the first product uses $H^1\cdot H^1\subset L^2$.
The term $\psi\Delta q$ in the nutrient equation satisfies the same
estimates, with $\phi'(W+\psi)$ replaced by $1$.

The terms $U\Theta$ and $U\psi$ have lower spatial order. For example,
\begin{align*}
 \|U\Theta\|_{L^2H^2}
       &\le C\sqrt T\,\|\psi\|_{C H^3},\\*
 \|\partial_t(U\Theta)\|_{L^2L^2}
       &=\|U\phi'(W+\psi)\psi_t\|_{L^2L^2}
         \le C\sqrt T\,\|\psi_t\|_{C H^1}.
\end{align*}
The corresponding bounds for $U\psi$ are immediate. Together,
these estimates prove the first inequality below:
\begin{equation}\label{eq:local-nonlinear}
\begin{aligned}
 \|\mathcal N(Q)\|_{\mathcal G_T}&\le C(R+\sqrt T)R,\\
 \|\mathcal N(Q)-\mathcal N(\widetilde Q)\|_{\mathcal G_T}
       &\le C(R+\sqrt T)\|Q-\widetilde Q\|_{\mathcal X_T}.
\end{aligned}
\end{equation}
For the second, separate a typical difference as
\[
 (a-a_*)\Delta q-(\widetilde a-a_*)\Delta\widetilde q
  =(a-a_*)\Delta(q-\widetilde q)
       +(a-\widetilde a)\Delta\widetilde q,
\]
and use
\[
 a-\widetilde a=(\psi-\widetilde\psi)
       \int_0^1\phi'\bigl(W+\widetilde\psi
                       +\theta(\psi-\widetilde\psi)\bigr)\dd\theta.
\]
The same product bounds apply in $L^2H^2$ and after time
differentiation. Each principal term contains one factor bounded by
$R$ and one difference; the lower-order terms contain the factor
$\sqrt T$.

On the closed ball in $\mathcal X_T$ with $Q(0)=Q_0$, solve the
linear problem with $G=\mathcal N(Q)$. Its initial derivative is
$\mathcal F(Q_0)$, whose nutrient trace is zero by compatibility.
To see that the ball is nonempty, consider the homogeneous frozen
solution. Its initial derivative has the required nutrient trace since
\[
 (-\mathcal A_*Q_0)_2|_{\partial\Omega}
       =(D\Delta\psi_0+b\Delta q_0)|_{\partial\Omega}=0.
\]
Moreover, $\|\mathcal A_*Q_0\|_{H^1}\le C\|Q_0\|_{H^3}$, so
\eqref{eq:local-linear} bounds its $\mathcal X_T$ norm by
$C\|Q_0\|_{H^3}$. For sufficiently small data this is below $R$.
For two elements of this ball, $\mathcal N(Q)(0)$ and
$\mathcal N(\widetilde Q)(0)$ agree. The linear difference estimate
therefore has zero initial state and zero initial derivative.
Writing $\mathcal T$ for the resulting solution map, we have
\begin{align*}
 \|\mathcal TQ\|_{\mathcal X_T}
       &\le C_0\|Q_0\|_{H^3}+C_1(R+\sqrt T)R,\\
 \|\mathcal TQ-\mathcal T\widetilde Q\|_{\mathcal X_T}
       &\le C_1(R+\sqrt T)\|Q-\widetilde Q\|_{\mathcal X_T}.
\end{align*}
Choose $R$ small, then $T_{\rm loc}$ small, so that
$C_1(R+\sqrt{T_{\rm loc}})<1/2$. Finally choose $\rho_{\rm loc}$
with $C_0\rho_{\rm loc}<R/2$. Banach's theorem gives a solution.
The bound
\[
 \|\mathcal F(Q_0)-\mathcal F(\widetilde Q_0)\|_{H^1}
       \le C\|Q_0-\widetilde Q_0\|_{H^3}
\]
and the same absorption give Lipschitz dependence on the initial data.
For any other solution with the same data,
$Q_t(0)=\mathcal F(Q_0)$. After decreasing the data radius, continuity
at zero makes $\|Q\|_{C H^3}+\|Q_t\|_{C H^1}<R/2$ on a short
interval. The time-integrated terms in \eqref{eq:construction-norm}
tend to zero with the interval length, so their sum is also below
$R/2$ after a further reduction. The competing solution therefore
belongs to the contraction ball. Repeating the argument proves
uniqueness while the state remains small.

The equation gives $Q_t(t)=\mathcal F(Q(t))$ in $H^1$ at every time.
The nutrient trace vanishes almost everywhere and is continuous in
time, so it vanishes at every time. Compatibility is therefore
preserved. The common local time permits a restart whenever
$\|Q(t)\|_{H^3}<\rho_{\rm loc}/2$, proving
\eqref{eq:continuation}. All constants are uniform under the common
ellipticity and stationary coefficient bounds for nearby masses.
\end{proof}

\Needspace{7\baselineskip}
\section{The interval problem at lower regularity}\label{app:interval}
\begin{theorem}[Interval stability and mass selection]\label{thm:interval}
Assume the structural and $C^3$ conditions in Assumption~\ref{ass:phi}.
For each $m>0$ there is a finite $D_1(m)\ge m/4$ such that,
if $D>D_1(m)$, there are $r_m,\delta,C,\omega>0$ with the
following property. Let $u_0\in L^2(I)$, $w_0-b\in H_0^1(I)$, and put
\[
 d_0=\|u_0-U_m\|_2+\|w_0-W_m\|_{H^1}<\delta,
 \qquad \mu_0=\int_Iu_0.
\]
Then $\mu_0\in I_m=[m-r_m,m+r_m]\Subset(0,\infty)$, and
\eqref{eq:model} has a unique global solution such that
\begin{equation}\label{eq:interval-class}
\begin{aligned}
 u-U_{\mu_0}&\in C([0,T];L^2(I))\cap L^2(0,T;H^1(I)),\\
 w-W_{\mu_0}&\in C([0,T];H_0^1(I))\cap L^2(0,T;H^2(I))
\end{aligned}
\end{equation}
for every $T>0$. The nutrient equation holds in $L^2(0,T;L^2)$ and
the population equation, including zero flux, has the variational meaning
\begin{equation}\label{eq:interval-weak-flux}
 u_t\in L^2(0,T;(H^1(I))'),\qquad
 \langle u_t,v\rangle=-\int_I[u\phi(w)]_xv_x,\quad v\in H^1(I).
\end{equation}
Moreover, $b/4\le w\le5b/4$ and
\begin{equation}\label{eq:interval-decay}
 \|u(t)-U_{\mu_0}\|_2+\|w(t)-W_{\mu_0}\|_{H^1}
 \le Ce^{-\omega t}d_0.
\end{equation}
The initial mass and the selected profile satisfy the recentering bound
\[
 |\mu_0-m|+\|U_{\mu_0}-U_m\|_2+\|W_{\mu_0}-W_m\|_{H^1}
 \le Cd_0.
\]
The selection map is locally Lipschitz in $L^2\times H^1$; the constants
can be chosen uniformly for nearby reference masses. Nonnegative initial
population remains nonnegative.
\end{theorem}

\begin{proof}
Fix the actual mass $m$ and write $(U,W)=(U_m,W_m)$. Set
\begin{equation}\label{eq:interval-system}
\begin{aligned}
 P(x,t)&=\int_0^x(u(y,t)-U(y))\dd y,\qquad \psi=w-W,\\
 P_t&=[(U+P_x)\phi(W+\psi)]_x,\\
 \psi_t&=D\psi_{xx}-U\psi-WP_x-\psi P_x,
 \qquad P|_{\partial I}=\psi|_{\partial I}=0.
\end{aligned}
\end{equation}
The initial primitive belongs to $H_0^1$ and
$\|P_0\|_{H^1}\le C\|u_0-U\|_2$.
Use the trace space $X=H_0^1(I)^2$ and
\[
 \mathbb E_T=H^1(0,T;L^2(I)^2)
       \cap L^2(0,T;(H^2(I)\cap H_0^1(I))^2)
       \hookrightarrow C([0,T];X),
\]
with the sum norm including the continuous trace norm. Linearize at zero:
\[
 \mathcal A\binom P\psi
 =\binom{[\phi(W)P_x+U\phi'(W)\psi]_x}
              {D\psi_{xx}-U\psi-WP_x}.
\]
Its principal matrix is diagonal and uniformly positive; all off-diagonal
terms are of first or zeroth order. Dirichlet maximal regularity
\cite{DenkHieberPruss,Solonnikov1965} gives the linear solution estimate
\begin{equation}\label{eq:interval-MR}
 \|h\|_{\mathbb E_T}\le C(\|h_0\|_X+\|h_t-\mathcal Ah\|_{L^2L^2}),
 \qquad 0<T\le1.
\end{equation}
Testing with $-P_{xx}$ and $-\psi_{xx}$, and estimating the
lower-order terms by Young's inequality, gives
$\frac{\dd}{\dd t}\|h_x\|_2^2+c\|h_{xx}\|_2^2
\le C\|h_x\|_2^2+C\|F\|_2^2$; the equations control $h_t$.
Gronwall's inequality gives a constant uniform for $T\le1$.

The nonlinear remainder is
\[
 \mathcal N(h)=\binom{[P_x(\phi(W+\psi)-\phi(W))
       +U(\phi(W+\psi)-\phi(W)-\phi'(W)\psi)]_x}{-\psi P_x}.
\]
On a ball $\|h\|_{\mathbb E_T}\le R$ small enough to keep
$W+\psi\in I_b$, Taylor's formula and $H^1(I)\hookrightarrow L^\infty(I)$
give
\begin{align*}
 \|\mathcal N(h)\|_{L^2L^2}&\le C(1+\sqrt T)R^2,\\
 \|\mathcal N(h)-\mathcal N(\widetilde h)\|_{L^2L^2}
 &\le C(1+\sqrt T)R\|h-\widetilde h\|_{\mathbb E_T}.
\end{align*}
For example, the derivative product obeys
$\|\psi_xP_x\|_{L^2L^2}
\le C\|\psi\|_{C H^1}\|P\|_{L^2H^2}$; all other terms have the
same or lower order. Choose the common nonlinear constant first, then
$R,T_{\rm loc}$ with $C_{\rm MR}C(1+\sqrt{T_{\rm loc}})R<1/2$,
and finally a trace radius $\delta_{\rm loc}<R/(4C_{\rm MR})$.
The linear solution map with forcing $\mathcal N(h)$ is a contraction.
It gives existence, uniqueness, and a common restart time whenever the
trace norm stays below $\delta_{\rm loc}$.

The stationary bounds and \eqref{eq:interval-derivative} give
$\|U\|_\infty\le C$, $b/2\le W\le b$, and
$\|U_x\|_\infty+\|W_x\|_\infty\le C/D$, with constants independent
of $D\ge m/4$. On a bootstrap interval assume
$\|P\|_{H^1}+\|\psi\|_{H^1}\le\delta_*$ and
$\|\psi\|_\infty<b/4$. Put $a=\phi(W+\psi)$ and
$\Theta=\phi(W+\psi)-\phi(W)$.
Testing the primitive equation with $-P_{xx}$ gives
\[
 \frac12\frac{\dd}{\dd t}\|P_x\|_2^2+\int_IaP_{xx}^2
 =-\int_I[U\Theta]_xP_{xx}
   -\int_IP_x\phi'(W+\psi)(W_x+\psi_x)P_{xx}.
\]
We treat the three terms on the right separately. Since
\[
 \Theta_x=\phi'(W+\psi)\psi_x
                 +[\phi'(W+\psi)-\phi'(W)]W_x,
\]
the coefficient bounds, Poincar\'e's inequality for $\psi$, and
Young's inequality imply
\begin{align*}
 \|[U\Theta]_x\|_2
       &\le C\bigl(\|\psi_x\|_2+\|\psi\|_2\bigr)
        \le C\|\psi_x\|_2,\\
 \left|\int_I[U\Theta]_xP_{xx}\right|
       &\le\frac{a_b}{8}\|P_{xx}\|_2^2+C\|\psi_x\|_2^2.
\end{align*}
Next, $\int_I P_x=0$, so $\|P_x\|_2\le C\|P_{xx}\|_2$ and
\[
 \left|\int_I P_x\phi'(W+\psi)W_xP_{xx}\right|
       \le\frac{C}{D}\|P_x\|_2\|P_{xx}\|_2
       \le\frac{C}{D}\|P_{xx}\|_2^2.
\]
For the term containing $\psi_x$, use the interpolation inequality
\[
 \|P_x\|_\infty\le C\|P_x\|_2^{1/2}\|P_{xx}\|_2^{1/2}.
\]
Consequently,
\begin{align*}
 \left|\int_I P_x\phi'(W+\psi)\psi_xP_{xx}\right|
 &\le C\|\psi_x\|_2\|P_x\|_2^{1/2}\|P_{xx}\|_2^{3/2}\\*
 &\le\frac{a_b}{8}\|P_{xx}\|_2^2
                  +C\|\psi_x\|_2^4\|P_x\|_2^2\\*
 &\le\left(\frac{a_b}{8}+C\delta_*^4\right)\|P_{xx}\|_2^2.
\end{align*}
Here Young's inequality was used with exponents $4/3$ and $4$.
Since $\int_I\psi_x=0$, Poincar\'e also gives
$\|\psi_x\|_2\le C\|\psi_{xx}\|_2$. Choose $D$ large and then
$\delta_*$ small to retain a fixed $c_1>0$. We obtain
\begin{equation}\label{eq:interval-P-energy}
 \frac12\frac{\dd}{\dd t}\|P_x\|_2^2+c_1\|P_{xx}\|_2^2
       \le C_2\|\psi_{xx}\|_2^2.
\end{equation}

Testing the nutrient equation with $-\psi_{xx}$ gives
\[
 \frac12\frac{\dd}{\dd t}\|\psi_x\|_2^2+D\|\psi_{xx}\|_2^2
       =\int_I U\psi\psi_{xx}
          +\int_I WP_x\psi_{xx}
          +\int_I\psi P_x\psi_{xx}.
\]
For these three integrals, the estimates are
\begin{align*}
 \left|\int_I U\psi\psi_{xx}\right|
 &\le C\|\psi\|_2\|\psi_{xx}\|_2
  \le\left(\frac D8+\frac CD\right)\|\psi_{xx}\|_2^2,\\
 \left|\int_I WP_x\psi_{xx}\right|
 &\le b\|P_x\|_2\|\psi_{xx}\|_2
  \le\frac D8\|\psi_{xx}\|_2^2+\frac CD\|P_{xx}\|_2^2,\\
 \left|\int_I\psi P_x\psi_{xx}\right|
 &\le C\delta_*\|P_x\|_2\|\psi_{xx}\|_2
  \le\frac D8\|\psi_{xx}\|_2^2
                    +\frac{C\delta_*^2}{D}\|P_{xx}\|_2^2.
\end{align*}
Thus, for a fixed common $C_3$,
\begin{equation}\label{eq:interval-psi-energy}
 \frac12\frac{\dd}{\dd t}\|\psi_x\|_2^2
       +\left(\frac D2-\frac{C_3}{D}\right)\|\psi_{xx}\|_2^2
 \le\frac{C_3(1+\delta_*^2)}{D}\|P_{xx}\|_2^2.
\end{equation}
The energy identity in $\mathbb E_T$ justifies these tests with
$H_0^1$ initial traces.
Fixing the constants above, choose $D_1(m)$ large enough that
for $D>D_1(m)$,
\[
 c_1-\frac{C_3(1+\delta_*^2)}D>0,\qquad
 \frac D2-\frac{C_3}D-C_2>0.
\]
Adding \eqref{eq:interval-P-energy} and \eqref{eq:interval-psi-energy}
then gives $E_1'+c_0\mathcal D_1\le0$, where
\[
 E_1=\|P_x\|_2^2+\|\psi_x\|_2^2,\qquad
 \mathcal D_1=\|P_{xx}\|_2^2+\|\psi_{xx}\|_2^2
                         \ge\pi^2E_1.
\]
Hence $E_1(t)\le E_1(0)e^{-c_0\pi^2t}$.
Choose the initial radius so that this estimate keeps the $H^1$ norm
below both the bootstrap and local restart radii, and keeps
$\|\psi\|_\infty<b/4$. Continuity improves the bootstrap strictly;
the common restart time gives global existence.
Since $u-U=P_x$, this proves the equal-mass decay and
\eqref{eq:interval-class}.

Let $F=u\phi(w)$. The primitive equation gives $P_t=F_x\in L^2L^2$,
so differentiating weakly in $x$ yields \eqref{eq:interval-weak-flux}.
Conversely, for a physical solution in \eqref{eq:interval-class} in a
positive nutrient range,
$F_x=\phi(w)u_x+u\phi'(w)w_x\in L^2L^2$, since
$u\in L^2H^1$ and $w\in C H^1$. Testing
\eqref{eq:interval-weak-flux} with
$v_\eta(x)=\int_x^1\eta(s)\dd s$ gives $P_t=F_x$ by Fubini's identity.
The nutrient equation similarly gives $\psi_t\in L^2L^2$.
Thus the primitive lies in $\mathbb E_T$, and local uniqueness, repeated
on overlapping intervals, proves uniqueness in the stated physical class.

For uniformity, first take a compact mass interval around $m$ on which
$D>\mu/4$. The difference equation \eqref{eq:steady-difference} on $I$
gives Lipschitz dependence in $W^{2,\infty}$ for $W_\mu$ and in
$W^{1,\infty}$ for $U_\mu$. Keep the Young parameters and $\delta_*$
fixed in the preceding estimates. Their coefficients can be chosen as
continuous functions of these profile norms, so the two strict coercivity
inequalities remain valid on a smaller interval $I_m$. The linear estimates,
nonlinear constants, and restart radius are uniform there, giving
a common small-data radius and decay rate.

For data of mass $\mu_0$ near the reference state, use
$|\mu_0-m|\le\|u_0-U_m\|_2$ and
\eqref{eq:mass-interval-bounds} to recenter at
$(U_{\mu_0},W_{\mu_0})$. The new primitive vanishes at both endpoints,
and its initial $H^1$ norm, together with that of $w_0-W_{\mu_0}$,
is at most $Cd_0$. The common fixed-mass estimate gives
\eqref{eq:interval-decay} and the asserted selection properties.
For nonnegative data, the negative-part estimate from the main proof
applies because $w_x\in L^2(0,T;L^\infty(I))$, making its Gronwall
coefficient integrable.
\end{proof}

\section{Numerical method and further results}\label{app:numerics}
We describe the spatial discretization, compatible perturbations, and
checks of relaxation and stationary asymptotics. Further $D=5$ runs
with larger mass changes and a different initial-data prescription are
provided in the separate supplementary numerical illustrations.

\subsection{Discretization and energy diagnostic}\label{app:numerical-method}
We use continuous piecewise-linear finite elements with nodal mass
lumping on polygonal approximations of the ellipse.
Starting from an eight-triangle fan on the unit disk, we repeatedly
bisect the edges, project new boundary vertices onto the unit circle,
and apply the affine map \((X,Y)\mapsto(X,0.65Y)\).
Let \(\mathsf M\) be the diagonal lumped mass matrix and
\(\mathsf K\) the symmetric stiffness matrix.
With \(\mathcal I\) denoting the interior nodes, the semidiscrete system is
\[
\begin{aligned}
 \mathsf M\dot{\mathbf u}
       &=-\mathsf K\mathbf z,\qquad
          \mathbf z=\mathbf u\phi(\mathbf w),\\
 \mathsf M_{\mathcal I}\dot{\mathbf w}_{\mathcal I}
       &=-D(\mathsf K\mathbf w)_{\mathcal I}
          -\mathsf M_{\mathcal I}(\mathbf u\mathbf w)_{\mathcal I},
\end{aligned}
\]
where products are componentwise and the boundary values of \(\mathbf w\)
are fixed to \(1\).
Since \(\mathsf K\mathbf1=0\), the discrete mass
\(\mathbf1^T\mathsf M\mathbf u\) is conserved by the semidiscrete equations.

For each mass \(\mu\), we compute the discrete equilibrium
\((Z_{\mu,h},W_{\mu,h},U_{\mu,h})\) from the stationary nutrient equation and
\[
 U_{\mu,h}\phi(W_{\mu,h})=Z_{\mu,h}\mathbf1,
 \qquad \mathbf1^T\mathsf M U_{\mu,h}=\mu.
\]

We integrate with the variable-order BDF method in SciPy's
\texttt{solve\_ivp}, using the analytic sparse Jacobian for implicit time stepping. The tolerances for each group of runs are
specified below. For related analyses of time-stepping schemes, see
\cite{Cheng2025BIT,Cheng2026IMA,ChengLiPromislowWetton2021}.

Write $\|a\|_{\mathsf M}^2=a^T\mathsf M a$ and use the corresponding
product norm for pairs.
For each trajectory, we solve the discrete Neumann problem
\[
 \mathsf Kq_{\mu,h}
 =\mathsf M(\mathbf u-U_{\mu,h}),\qquad
 \mathbf1^T\mathsf Mq_{\mu,h}=0,
\]
by the augmented system
\[
 \begin{pmatrix}
 \mathsf K&\mathsf M\mathbf1\\
 \mathbf1^T\mathsf M&0
 \end{pmatrix}
 \binom{q_{\mu,h}}{\lambda_h}
 =\binom{\mathsf M(\mathbf u-U_{\mu,h})}{0}.
\]
The multiplier is
$\lambda_h=\mathbf1^T\mathsf M(\mathbf u-U_{\mu,h})/
(\mathbf1^T\mathsf M\mathbf1)$; it vanishes for an exactly mass-compatible
state and removes the roundoff-level mean defect in the stored states.
We then evaluate
\[
 E_{0,h}(t)
 =q_{\mu,h}^T\mathsf Kq_{\mu,h}
  +\|\mathbf w-W_{\mu,h}\|_{\mathsf M}^2.
\]
The population part retains the exact semidiscrete identity
\[
 \frac12\frac{\dd}{\dd t}
 (q_{\mu,h}^T\mathsf Kq_{\mu,h})
 =-(\mathbf u-U_{\mu,h})^T
   \mathsf M(\mathbf z-Z_{\mu,h}\mathbf1).
\]
The diagnostic \(E_{0,h}\) approximates the low-order potential energy
$E_0$. We examine its decay and the solution errors under changes in
mesh size, solver tolerances, and initial amplitude.

\subsection{Compatible data and the diffusion bound}\label{app:compatible-data}
For \eqref{eq:compatible-numerical-data}, put $r^2=X^2+Y^2$ and
\[
 \eta=\begin{cases}
 \exp\bigl(1-(1-r^2/0.75^2)^{-1}\bigr),&r<0.75,\\
 0,&r\ge0.75.
 \end{cases}
\]
The bump $\eta$ is smooth and supported strictly inside the ellipse.
Symmetry gives $\int_\Omega\eta\,(X+0.4Y)=0$. Because the initial pair agrees
with $(U_1,W_1)$ throughout a boundary neighborhood,
\[
 \partial_\nu(u_0\phi(w_0))=0,\qquad
 w_0=b,\qquad D\Delta w_0-bu_0=0
 \quad\hbox{on }\partial\Omega.
\]
At fixed $D$, the norms of these bump functions are finite, and hence
\[
 \|u_0-U_1\|_{H^1}+\|w_0-W_1\|_{H^3}
 \le C\bigl(\varepsilon+|\mu-1|\bigr).
\]
On each mesh, we evaluate the bumps at the nodes, replace their
integrals by lumped quadrature, and remove the roundoff-level mean of
the odd term using the same bump. The final mass correction is also a
multiple of this interior bump. The discrete perturbations therefore
vanish at every node in the boundary collar. The $H^1\times H^3$
estimate above concerns the smooth continuous data; the finite element
initial values are their continuous piecewise-linear nodal approximations.

The geometric quantities in \eqref{eq:diffusion-threshold} have useful explicit
bounds here. The area is $0.65\pi$, and
\[
 \tau(x,y)=\frac{1-x^2-(y/0.65)^2}{2(1+0.65^{-2})}.
\]
The ellipse lies inside $(-1,1)\times(-0.65,0.65)$. Domain monotonicity
for the Dirichlet eigenvalue therefore gives
\[
 \lambda_1(\Omega)\ge\frac{\pi^2}{4}(1+0.65^{-2})>8.3074.
\]
With $b=1$, $a_b=1-e^{-1/4}$, and $L_b=e^{-1/4}$, substitution in
\eqref{eq:constants}--\eqref{eq:diffusion-threshold} gives
\[
 \frac{25}{2a_b}\le D_3(\mu)<57.0865
 \quad (0.999\le\mu\le1.001),\qquad
 \frac{25}{2a_b}=56.5101458\ldots.
\]
All upper-bound expressions are increasing in $\mu$, so the right-hand
bound follows by using $\mu=1.001$. In particular $D=60$ satisfies the
diffusion condition throughout this interval, whereas $D=5$ does not.

\subsection{Decay and refinement at large diffusivity}\label{app:compatible-results}
For $D=60$, take $\varepsilon=10^{-3}$ and
$\mu\in\{0.999,1,1.001\}$. The relative and absolute BDF tolerances
are $10^{-11}$ and $10^{-13}$, respectively. We store 255
times on $[0,5]$, combining a uniform grid with additional points in
the initial transient. On the main mesh, stationary nutrient residuals
in the unscaled nodal equations are below $1.08\times10^{-13}$.
The three sampled trajectories retain positive population and satisfy
$0.9987<w_h\le1$; their sampled mass drift is below
$2.1\times10^{-14}$.

Figure~\ref{fig:compatible-snapshots} shows the mass-$1$ spatial
relaxation. At $t=0,0.1,0.5,1$, its relative maximum population
amplitudes are $1$, $0.180$, $0.0739$, and $0.0245$.

\begin{figure}[htbp]
 \centering
 \includegraphics[width=\textwidth]{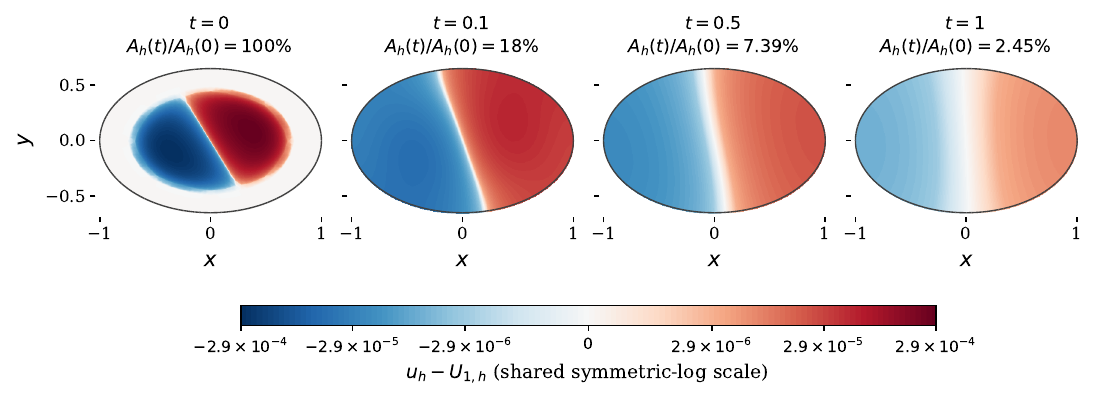}
 \caption{Nonradial relaxation for the compatible data
 \eqref{eq:compatible-numerical-data} with $D=60$, $\mu=1$, and
 $\varepsilon=10^{-3}$. The shared symmetric-log scale has limits
 $\pm2.9\times10^{-4}$ and linear core
 $|u_h-U_{1,h}|\le2.9\times10^{-6}$.
 The amplitude labels report decay without panelwise rescaling.}
 \label{fig:compatible-snapshots}
\end{figure}

Figure~\ref{fig:compatible-comparison}(a) compares the solutions with
their mass-selected equilibria and with the mass-$1$ reference state.
The selected-equilibrium errors decay below $10^{-9}$, while the
unequal-mass reference errors remain near $7.00\times10^{-4}$.
For the compatible mass-$1$ family, panel (b) compares $D=60$ with
$D=5$, recomputing the reference equilibrium at each diffusivity.
The stationary nutrient minima are approximately $0.998786$ and
$0.985555$, respectively. At the larger diffusivity, the profile is
closer to the reservoir concentration and the initial nutrient transient
is shorter. The fitted low-order energy rates on $1\le t\le4$ are
$4.396$ and $4.375$.

\begin{figure}[htbp]
 \centering
 \includegraphics[width=\textwidth]{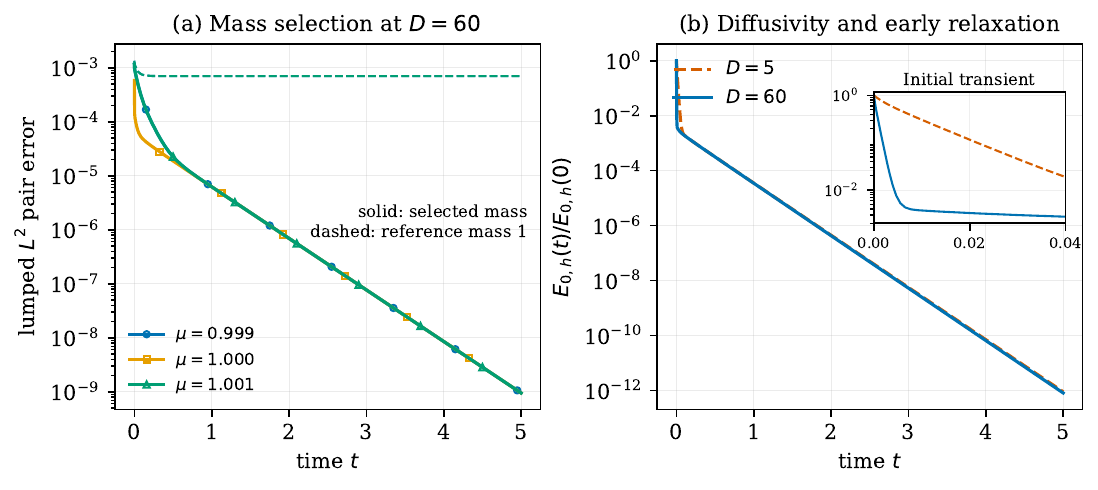}
 \caption{Compatible perturbations with $\varepsilon=10^{-3}$.
 (a) Selected-equilibrium errors (solid) and errors relative to the
 mass-$1$ reference (dashed) at $D=60$. The curves for nearby masses
 almost coincide. (b) Normalized low-order potential energy for mass
 $1$ at two diffusivities; the inset resolves the initial transient.
 The sufficient diffusion condition is verified for $D=60$.}
 \label{fig:compatible-comparison}
\end{figure}

Reducing the amplitude and tightening the time tolerances leave the
fitted decay rate nearly unchanged. For $\mu=1$ and $D=60$, amplitudes
$10^{-3}$, $5\times10^{-4}$, and $2.5\times10^{-4}$ give fitted
rates $4.3957, 4.3957, 4.3957$, with the relative time tolerance
reduced proportionally to the amplitude.
The initial energies divided by $\varepsilon^2$ agree to roundoff;
the smaller-amplitude rescaled trajectories differ from the principal
rescaled trajectory by at most $1.57\times10^{-6}$ in the sampled
lumped pair norm.
For the principal mass-$1$ run, tightening the relative/absolute
tolerances from $10^{-11}/10^{-13}$ to $10^{-13}/10^{-15}$ changes the
sampled solution by at most $2.93\times10^{-11}$ in the pair norm and
the fitted rate by less than $3.2\times10^{-4}$.

Table~\ref{tab:compatible-check} uses a 16641-node reference at
$t=0.5$ and relative/absolute tolerances $10^{-13}/10^{-15}$.
After restriction to common nodes, $e_{\rm full}$ compares full states
and $e_{\rm pert}$ compares $(u_h-U_{1,h},w_h-W_{1,h})$, both in the
coarse lumped pair norm. The former includes polygonal geometry and
recomputed equilibria. On each mesh, $e_{\rm tol}$ compares relative
tolerances $10^{-8}$ and $10^{-13}$, with absolute tolerance one
hundredth as large; $\delta M$ is the maximum sampled mass drift
in those two runs.

\begin{table}[htbp]
 \centering
 \caption{Refinement checks for the compatible mass-$1$ data at
 $D=60$, $\varepsilon=10^{-3}$, and $t=0.5$.}
 \label{tab:compatible-check}
 \small
 \begin{tabular}{rrrrrr}
 \toprule
 Nodes & $h_{\max}$ & $e_{\rm full}$ & $e_{\rm pert}$ & $e_{\rm tol}$ & $\delta M$\\
 \midrule
 81 & 0.3016 & $4.49\times10^{-3}$ & $7.65\times10^{-7}$ & $1.49\times10^{-9}$ & $1.11\times10^{-15}$\\
 289 & 0.1546 & $1.11\times10^{-3}$ & $3.05\times10^{-7}$ & $8.36\times10^{-10}$ & $1.89\times10^{-15}$\\
 1089 & 0.0782 & $2.64\times10^{-4}$ & $1.33\times10^{-8}$ & $1.36\times10^{-9}$ & $2.00\times10^{-15}$\\
 4225 & 0.0393 & $5.28\times10^{-5}$ & $3.50\times10^{-9}$ & $9.69\times10^{-10}$ & $1.02\times10^{-14}$\\
 \bottomrule
 \end{tabular}
\end{table}
Both spatial errors decrease under refinement, with the empirical
orders varying across the coarse levels.

\subsection{Checks of the stationary expansion}\label{app:expansion-numerics}
For the two laws in Section~\ref{sec:numerical-profiles}, the ratio
$g(w)=w/\phi(w)$ is affine: it is $1$ for $\phi(w)=w$ and
$(1+w)/2$ for $\phi(w)=2w/(1+w)$. Writing $v_h=1-W_{D,h}$,
the fixed-flux nutrient solve is therefore
\[
 (D\mathsf K_{\mathcal I\mathcal I}
       +Z_{D,h}g'\mathsf M_{\mathcal I})v_{h,\mathcal I}
 =Z_{D,h}\mathsf M_{\mathcal I}\mathbf1,
 \qquad v_h|_{\partial\Omega_h}=0,
\]
with $g'=0$ or $1/2$. A scalar root solve enforces
$\mathbf1^T\mathsf M U_{D,h}=1$. Solving for the deficit avoids
cancellation in the stiffness term for nearly constant nutrient profiles.

Let $\tau_h$ solve the discrete Dirichlet torsion problem
$\mathsf K_{\mathcal I\mathcal I}\tau_{h,\mathcal I}
=\mathsf M_{\mathcal I}\mathbf1$, and put
\[
 \rho_h=\frac1{\mathbf1^T\mathsf M\mathbf1},\qquad
 \langle\tau_h\rangle_h=
 \frac{\mathbf1^T\mathsf M\tau_h}{\mathbf1^T\mathsf M\mathbf1},\qquad
 A_{W,h}=\rho_h\tau_h,\quad
 A_{U,h}=\rho_h^2\kappa_1(\tau_h-\langle\tau_h\rangle_h).
\]
We measure the fixed-mesh scaled remainders
\begin{equation}\label{eq:discrete-expansion-errors}
 e_{W,h}=\|D(1-W_{D,h})-A_{W,h}\|_{\mathsf M},\qquad
 e_{U,h}=\|D(U_{D,h}-\rho_h)-A_{U,h}\|_{\mathsf M}.
\end{equation}
Centering at the continuum mean instead would introduce the term
$D(\rho_h-\rho)$, which grows with $D$ at fixed mesh.

On 66049 nodes, the observed orders of $e_{W,h},e_{U,h}$ under
doubling $D$ lie between $0.9995$ and $1.0000$.
The quantities $D e_{W,h}$ and $D e_{U,h}$, corresponding to
$D^2$ times the original remainders, are nearly constant:
\begin{table}[htbp]
 \centering
 \caption{Rescaled stationary remainders on 66049 nodes.}
 \label{tab:expansion-remainders}
 \small
 \begin{tabular}{crrrr}
 \toprule
 &\multicolumn{2}{c}{$\phi(s)=s$}
 &\multicolumn{2}{c}{$\phi(s)=2s/(1+s)$}\\
 $D$&$D e_{W,h}$&$D e_{U,h}$&$D e_{W,h}$&$D e_{U,h}$\\
 \midrule
 60 & $2.1809\times10^{-3}$ & $2.7605\times10^{-4}$ & $2.5961\times10^{-3}$ & $1.3594\times10^{-4}$\\
 120 & $2.1813\times10^{-3}$ & $2.7605\times10^{-4}$ & $2.5969\times10^{-3}$ & $1.3594\times10^{-4}$\\
 240 & $2.1815\times10^{-3}$ & $2.7605\times10^{-4}$ & $2.5974\times10^{-3}$ & $1.3594\times10^{-4}$\\
 \bottomrule
 \end{tabular}
\end{table}

At $D=240$, the contrast
$D\,\operatorname{osc}(U_{D,h})/
 (\rho_h^2\operatorname{osc}(\tau_h))$ is
$1.000000$ and $0.500018$, consistent with
$\kappa_1=1$ and $1/2$.

To measure the error in the leading coefficients, we use the analytic
values on the ellipse. Here $\rho=(0.65\pi)^{-1}$ and the explicit torsion
function in Appendix~\ref{app:compatible-data} has
$\langle\tau\rangle=\|\tau\|_\infty/2$.
We therefore compare $A_{W,h},A_{U,h}$ with the nodal values of
$\rho\tau$ and $\rho^2\kappa_1(\tau-\langle\tau\rangle)$.
Table~\ref{tab:expansion-coefficients} reports the lumped errors.
Both laws share the nutrient coefficient; the saturating law has
half the displayed population error.

\begin{table}[htbp]
 \centering
 \caption{Convergence of the discrete leading coefficients to their
 analytic continuum values.}
 \label{tab:expansion-coefficients}
 \small
 \begin{tabular}{rrrr}
 \toprule
 Nodes&$h_{\max}$&Nutrient coefficient error&Population coefficient error\\
 \midrule
 1089 & 0.0782 & $5.20\times10^{-5}$ & $4.50\times10^{-5}$\\
 4225 & 0.0393 & $1.30\times10^{-5}$ & $1.13\times10^{-5}$\\
 16641 & 0.0197 & $3.24\times10^{-6}$ & $2.81\times10^{-6}$\\
 66049 & 0.0099 & $8.10\times10^{-7}$ & $7.03\times10^{-7}$\\
 \bottomrule
 \end{tabular}
\end{table}

The errors decrease by approximately a factor of four at each refinement.
On the finest mesh the population coefficient error is still comparable
to the $D=240$ scaled expansion remainder, so spatial refinement remains
relevant at the largest diffusivity.

The maximum mass residual is below $5\times10^{-15}$.
After division by the lumped mass weights, the stationary nutrient
residual is below $5\times10^{-11}$ on the finest mesh.
Tightening the scalar root tolerances from $5\times10^{-15}/10^{-14}$
to $5\times10^{-16}/10^{-15}$ (absolute/relative) changes the scaled
population and nutrient profiles by less than $9\times10^{-14}$
in the lumped norm.

For $\phi(w)=w$, the ellipse admits an exact solution.
With $\tau_0=\|\tau\|_\infty$ and $c_D=1-\exp(-\rho\tau_0/D)$,
integration of the mass constraint gives
\[
 Z_D=\frac{Dc_D}{\tau_0},\qquad
 W_D=1-\frac{c_D}{\tau_0}\tau,\qquad U_D=\frac{Z_D}{W_D}.
\]
Here $\tau/\tau_0=1-r^2$ in elliptical radial coordinates, so
$\int_\Omega U_D=(D|\Omega|/\tau_0)[-\log(1-c_D)]=1$.
One-dimensional quadrature at $D=240$ gives continuum scaled
remainders of $9.090\times10^{-6}$ for nutrient and $1.149\times10^{-6}$
for population in $L^2(\Omega)$. Both decrease with order approximately
one under doubling $D$.

\section*{Acknowledgments}
X. Cheng is supported in part by the NSFC (Grant Nos. 12401270 and
42450192), the Natural Science Foundation of Shanghai (Grant No. 24ZR1404200),
and the Shanghai Magnolia Talent Plan Pujiang Project (Grant No. 24PJA007).
X. Song is supported in part by the Youth S \& T Talent Support Programme of Guangdong Provincial Association for Science and Technology (Grant No. SKXRC2026776).

\paragraph{AI assistance}
The authors used OpenAI Codex in preparing this article and its
supplementary material, including assistance with drafting and revision, checking of mathematical arguments and checking of code for numerical computations and figures.
The authors assume responsibility for all content.

\end{document}